\documentclass[11pt,twoside]{article}
\usepackage{mathrsfs}
\usepackage{amsmath,amssymb}
\usepackage{amsfonts}
\usepackage{latexsym,bm}
\usepackage{amsthm}
\usepackage{amssymb,amscd}
\usepackage{amsbsy}
\usepackage{fancyhdr,graphicx}
\usepackage[dvips]{psfrag}
\usepackage{indentfirst}
\usepackage{graphics,color}
\usepackage{epsfig}
\usepackage{subfigure}
\usepackage{xcolor}

\numberwithin{equation}{section}

\newtheorem {theorem} {Theorem}%[section]
\newtheorem {proposition} [theorem]{Proposition}

\newtheorem {lemma}  [theorem]{Lemma}

\theoremstyle{remark}

\theoremstyle{definition}

\newtheorem {remark} [theorem]{Remark}

\begin{document}
\setlength{\parindent}{4ex}
\setlength{\parskip}{1ex}
\setlength{\oddsidemargin}{12mm}
\setlength{\evensidemargin}{9mm}

\title
{\textsc {On the number of limit cycles for a class of discontinuous quadratic differential systems}}
\begin{figure}[b]
\rule[-0.5ex]{7cm}{0.2pt}\\
\footnotesize $^{*}$Corresponding author. E-mail address:
cenxiuli2010@163.com (X. Cen); lism1983@126.com (S. Li); mcszyl@mail.sysu.edu.cn (Y. Zhao).
\end{figure}
\author
{{\textsc {Xiuli Cen$^{a,*}$, Shimin Li$^{b}$, Yulin Zhao$^{c}$}}\\[2ex]
{\footnotesize\it $^{a}$ Department of Mathematical Sciences, Tsinghua University, Beijing, 100084, P.R.China}\\
{\footnotesize\it $^{b}$ School of Mathematics and Statistics, Guangdong University of Finance and Economics,}\\
{\footnotesize\it Guangzhou, 510320, P.R.China}\\
{\footnotesize\it $^{c}$ Department of Mathematics, Sun Yat-sen University, Guangzhou, 510275, P.R.China}}
\date{}
\maketitle {\narrower \small \noindent {\bf Abstract\,\,\,} {The present paper is devoted to the study of
the maximum number of limit cycles bifurcated from the periodic orbits of the quadratic isochronous center
$\dot{x}=-y+\frac{16}{3}x^{2}-\frac{4}{3}y^{2},\dot{y}=x+\frac{8}{3}xy$ by the averaging method of first order, when it is perturbed inside a class of discontinuous quadratic polynomial differential systems. The \emph{Chebyshev
criterion} is used to show that this maximum number is 5 and can be realizable. The result and that in paper \cite{LC} completely answer the questions left in the paper \cite{LM}.}

}


\section{Introduction and statement of the main result}
It is well known that one of the important
open problems in the qualitative theory of real planar differential systems is the study of limit cycles.
For about one century,  scholars focus on the bifurcation of limit cycles in the continuous planar polynomial differential systems, see \cite{BL,CGP,CJ,GI,GV,LLT,LLLZ,MV,SV,Z} and the references therein. Nevertheless, it is still open even for the quadratic cases. In recent years, stimulated by the discontinuous phenomena in the real world, a great interest in the limit cycles of discontinuous planar polynomial differential systems has emerged, see for instance \cite{LC,LM,LNT} etc.

Recall that Loud first classified the quadratic polynomial differential system with an isochronous center into four kinds in \cite{L}:
\begin{equation*}\begin{array}{ll}
S_{1}:\dot{x}=-y+x^{2}-y^{2},\quad \dot{y}=x+2xy,\\[2ex]
S_{2}:\dot{x}=-y+x^{2},\quad \dot{y}=x+xy,\\[2ex]
S_{3}:\dot{x}=-y-\dfrac{4}{3}x^{2},\quad \dot{y}=x-\dfrac{16}{3}xy,\\[2ex]
S_{4}:\dot{x}=-y+\dfrac{16}{3}x^{2}-\dfrac{4}{3}y^{2},\quad
\dot{y}=x+\dfrac{8}{3}xy.
\end{array}\end{equation*}
Chicone and Jacobs proved that in \cite{CJ}, under all continuous quadratic polynomial perturbations, at most 1 limit cycle bifurcates from the periodic orbits of $S_1$, and at most 2 limit cycles
bifurcate from the periodic orbits of $S_2,S_3$ and $S_4$.
Llibre and Mereu studied the number of limit cycles bifurcated from the periodic orbits of quadratic isochronous centers $S_1$
and $S_2$ by the averaging method of first order, when they are perturbed inside a class of piecewise smooth quadratic polynomial systems \cite{LM}. They found that
at least 4 limit cycles can bifurcate from the periodic orbits of $S_1$, and at least 5 limit cycles can bifurcate from the periodic orbits of $S_2$, both of which show that the discontinuous systems have more limit cycles than the continuous ones.
Li and Cen obtained in \cite{LC} that there are at most 4 limit cycles bifurcating from the periodic orbits of $S_3$ by the averaging method of first order and \emph{Chebyshev criterion}.

In the present paper, we study the perturbations of quadratic isochronous center $S_{4}$:
\begin{equation}
\left(\begin{array}{ll}\dot{x}\\[2ex] \dot{y}\end{array}\right)=\label{eq1}\left\{\begin{array}{ll}
\left(\begin{array}{ll}-y+\dfrac{16}{3}x^{2}-\dfrac{4}{3}y^{2}+\varepsilon P_{1}(x,y)\\
x+\dfrac{8}{3}xy+\varepsilon Q_{1}(x,y)\end{array}\right), & \mbox{ $x> 0$,}
\\[2ex]
\left(\begin{array}{ll}-y+\dfrac{16}{3}x^{2}-\dfrac{4}{3}y^{2}+\varepsilon P_{2}(x,y)\\
x+\dfrac{8}{3}xy+\varepsilon Q_{2}(x,y)\end{array}\right), & \mbox{ $x< 0$,}
\end{array} \right.
\end{equation}\label{S4}
where $0<|\varepsilon|\ll1$ and $P_{i}(x,y),Q_{i}(x,y),i=1,2$ are quadratic polynomials in the variables $x$ and $y$, given by
\begin{equation}\label{PQ}\begin{split}
P_{1}(x,y)&=a_{10}x+a_{01}y+a_{20}x^{2}+a_{11}xy+a_{02}y^{2},\\[2ex]
Q_{1}(x,y)&=b_{10}x+b_{01}y+b_{20}x^{2}+b_{11}xy+b_{02}y^{2},\\[2ex]
P_{2}(x,y)&=c_{10}x+c_{01}y+c_{20}x^{2}+c_{11}xy+c_{02}y^{2},\\[2ex]
Q_{2}(x,y)&=d_{10}x+d_{01}y+d_{20}x^{2}+d_{11}xy+d_{02}y^{2}.
\end{split}\end{equation}

The object of this paper is to give the least upper bound of the number of limit cycles of system \eqref{S4} which bifurcate from the periodic orbits of quadratic isochronous center $S_{4}$. The techniques we use mainly include the averaging theory of first order and \emph{Chebyshev criterion}. To apply the averaging method, we also generalize a theorem (see Theorem \ref{th:BL} or \cite{BL}) which provides an approach to transform the  polynomial differential system to a normal form, see Theorem \ref{th:G}. It is worth noting that, we use some skills in dealing with the averaged function obtained,  which is a linear combination of various elementary functions, such as square root functions and logarithmic function, and the first and second complete elliptic integrals. It is challenging
to obtain the sharp upper bound of the number of zeros of the averaged function, that is equivalent to the number of limit cycles of system \eqref{S4} bifurcated from the periodic orbits of quadratic isochronous center $S_{4}$. Inspired by the idea of \cite{GI}, we eliminate the logarithmic function first and do not cause the increase in the number of zeros. Using some appropriate transformations and \emph{Chebyshev criterion}, we give the result as follows by qualitative analysis and algebraic calculations.
\begin{theorem}\label{th:S4}
The maximum number of limit cycles of discontinuous quadratic polynomial differential system \eqref{S4} which bifurcate from the periodic orbits of the quadratic isochronous center $\eqref{S4}|_{\varepsilon=0}$ (i.e., $S_4$)
is 5 by the averaging method of first order.
\end{theorem}

Theorem \ref{th:S4} and Theorem 1.1 in the paper \cite{LC} completely answer the questions left in
Table 1 of the paper \cite{LM}, and thus, by the averaging method of first order, there are at least 4 limit cycles bifurcating from the periodic orbits of the isochronous centers $S_1$ and $S_3$, while at least 5 limit cycles can bifurcate from the periodic orbits of the isochronous centers $S_2$ and $S_4$, under piecewise smooth quadratic polynomial perturbations.
More importantly, Theorem \ref{th:S4} gives the exact upper bound of the number of limit cycles bifurcated from the periodic orbits of the quadratic isochronous center $S_4$, which is challenging.

The paper is organized as follows. In sections \ref{sec:pre}, we introduce the basic theory on averaging method, and provide a new technique transforming planar polynomial differential system to a specific form, which is an extension of the theorem \cite{BL}. The averaged function associated to system \eqref{S4} is  obtained in section \ref{sec:aver}. Section \ref{sec:proof} focuses on the analysis of the least upper bound for the number of zeros of the averaged function, and \emph{Chebyshev criterion} is used to prove the main result. Finally, we give some important and long expressions in Appendix for reference.

\section{Preliminary results}\label{sec:pre}
In this section we summarize the theorem of first order averaging
method for discontinuous differential
systems in \cite{LNT}. A more general introduction to averaging
method can be found in the book \cite{SV}.

\begin{theorem}\label{th:average}
\cite{LNT} Consider the following discontinuous differential
equation
\begin{equation}\label{x}
\dfrac{\mathrm{d}x}{\mathrm{d}t}=\varepsilon
F(t,x)+\varepsilon^{2}R(t,x,\varepsilon),
\end{equation}
with
\begin{equation}\begin{array}{ll}
F(t,x)=F_{1}(t,x)+\mathrm{sign}(h(t,x))F_{2}(t,x),\\[2ex]
R(t,x,\varepsilon)=R_{1}(t,x,\varepsilon)+\mathrm{sign}(h(t,x))R_{2}(t,x,\varepsilon),
\end{array}\end{equation}
where $F_{1}, F_{2}: \mathbb{R}\times D\rightarrow \mathbb{R}^n,\
R_{1}, R_{2}: \mathbb{R}\times D\times (-\varepsilon_{0},
\varepsilon_{0})\rightarrow \mathbb{R}^n$, and $h:\mathbb{R}\times
D\rightarrow \mathbb{R}$ are continuous functions, T-periodic in the
variable $t$ and $D$ is an open subset of $\mathbb{R}^n$.
$\mathrm{sign}(h)$ is the sign function defined by
\begin{equation*}\mathrm{sign}(h)=\left\{\begin{array}{ll}
1 & \mbox{if $h>0$,}\\
-1 & \mbox{if $h<0$.}
\end{array} \right.
\end{equation*}
We also suppose that $h$ is a $C^{1}$ function and has $0$ as a
regular value. Denote by $\mathcal{M}=h^{-1}(0)$, by
$\Sigma=\{0\}\times D \not\subseteq \mathcal{M}$, by
$\Sigma_{0}=\Sigma\setminus\mathcal {M}\neq \O $, and its elements
by $z\equiv (0,z)\notin \mathcal{M}$.

Define the averaged function $f:D\rightarrow \mathbb{R}^n$ as
\begin{equation}\label{f}
f(x)=\displaystyle\int_{0}^{T}F(t,x)\mathrm{d}t.
\end{equation}
We assume the following conditions:

$\mathrm{(i)}$ $F_{1}, F_{2}, R_{1}, R_{2}$ and $h$ are locally
$L$-Lipschitz with respect to $x$;

$\mathrm{(ii)}$ $\frac{\partial h}{\partial t}(t,x)\neq 0$  for each
$(t,x) \in \mathcal{M}$;

$\mathrm{(iii)}$ for $a\in \Sigma_{0}$ with $f(a)=0$, there exists a
neighborhood $V$ of $a$ such that $f(z)\neq 0$ for all $z\in
\bar{V}\setminus \{a\}$ and
$d_{B}(f,V,0)\neq 0$ which is the Brouwer degree function, and more details see Appendix A of \cite{BL}.

Then for $|\varepsilon|>0$ sufficiently small, there exists a
$T$-periodic solution $x(t,\varepsilon)$ of system \eqref{x} such
that $x(0,\varepsilon)\rightarrow a$ as $\varepsilon\rightarrow 0$.
\end{theorem}

\begin{remark}\label{re:average}
If $f$ is a $C^{1}$ function and the Jacobian $J_{f}(a)\neq 0$ , then the hypothesis (iii) holds, see \cite{BL}.
\end{remark}

Consider the planar differential system
\begin{equation}\begin{split}\label{eq:xyp}
\dot{x}&=P(x,y)+\varepsilon p(x,y),\\
\dot{y}&=Q(x,y)+\varepsilon q(x,y),
\end{split}\end{equation}
where $P(x,y)$, $Q(x,y)$, $p(x,y)$ and $q(x,y)$ are continuous functions, and $\varepsilon$ is a small parameter. Suppose that
system \eqref{eq:xyp}$|_{\varepsilon=0}$ has a continuous family of periodic orbits
\[
\Gamma_{h}:\ H(x,y)=h, \,\,h\in(h_c,h_s),
\]
where $H(x,y)$ is a first integral of \eqref{eq:xyp}$|_{\varepsilon=0}$, and $h_c$ and $h_s$ correspond to the center and
the separatrix polycycle, respectively.

The following theorem provides an approach to transform system \eqref{eq:xyp} to the form \eqref{x}.

\begin{theorem}\label{th:BL}
\cite{BL} Consider system \eqref{eq:xyp}$|_{\varepsilon=0}$ and its first integral $H=H(x,y)$. Assume that $xQ(x,y)-yP(x,y)\neq 0$ for all
$(x,y)$ in the period annulus. Let $\rho:(\sqrt{h_c},\sqrt{h_s})\times[0,2\pi)\rightarrow[0,\infty)$ be a continuous function such that
\begin{equation}
H(\rho(r,\theta)\cos\theta,\rho(r,\theta)\sin\theta)=r^2,
\end{equation}
for all $r\in(\sqrt{h_c},\sqrt{h_s})$ and all $\theta\in[0,2\pi)$. Then the differential equation which
describes the dependence between the square root of energy $r=\sqrt{h}$ and the angle $\theta$ for system \eqref{eq:xyp}
is
\begin{equation}
\dfrac{\mathrm{d}r}{\mathrm{d}\theta}=\varepsilon \dfrac{\mu(x^2+y^2)(Q p-Pq)}{2r(Qx-Py)+2r\varepsilon(qx-py)},
\end{equation}
where $\mu=\mu(x,y)$ is the integrating factor of system \eqref{eq:xyp}$|_{\varepsilon=0}$ corresponding to the first integral $H$, and $x=\rho(r,\theta)\cos\theta$ and $y=\rho(r,\theta)\sin\theta$.
\end{theorem}

The application of Theorem \ref{th:BL} is limit to isochronous centers $S_1$ and $S_2$, etc. We give a generalization of Theorem \ref{th:BL} as follows, which is especially applicable to isochronous centers $S_3$, $S_4$, and so on.

\begin{theorem}\label{th:G}
Consider system \eqref{eq:xyp}$|_{\varepsilon=0}$ and its first integral $H=H(x,y)$. If there exist differential homeomorphisms $X=h_1(x,y)$, $Y=h_2(x,y)$ such that $(X\widetilde{Q}-Y\widetilde{P})|_{x=g_1(X,Y),y=g_2(X,Y)}\neq 0$ for all
$(X,Y)$ in the period annulus, where $\widetilde{P}=\frac{\partial h_1}{\partial x}P+\frac{\partial h_1}{\partial y} Q$, $\widetilde{Q}=\frac{\partial h_2}{\partial x} P+\frac{\partial h_2}{\partial y} Q$, and $x=g_1(X,Y)$, $y=g_2(X,Y)$ are the inverse functions of $X=h_1(x,y)$, $Y=h_2(x,y)$. Let $\rho:(\sqrt{\mathcal{R}^{-1}(h_c)},\sqrt{\mathcal{R}^{-1}(h_s)})\times[0,2\pi)\rightarrow[0,\infty)$ be a continuous function such that
\begin{equation}
\widetilde{H}(\rho(r,\theta)\cos\theta,\rho(r,\theta)\sin\theta)=\mathcal{R}(r^2),
\end{equation}
for all $r\in(\sqrt{\mathcal{R}^{-1}(h_c)},\sqrt{\mathcal{R}^{-1}(h_s)})$ and all $\theta\in[0,2\pi)$, where $r=\sqrt{\mathcal{R}^{-1}(h)}$ is the inverse function of $h=\mathcal{R}(r^2)$ that increases with $r$, and $\widetilde{H}=\widetilde{H}(X,Y)=H|_{x=g_1(X,Y),y=g_2(X,Y)}$. Then the differential equation which
describes the dependence between the square root of energy $r=\sqrt{\mathcal{R}^{-1}(h)}$ and the angle $\theta$ for system \eqref{eq:xyp}
is
\begin{equation}
\dfrac{\mathrm{d}r}{\mathrm{d}\theta}=\varepsilon \dfrac{\mu(X^2+Y^2)(Qp-Pq)}{2r\mathcal{R}'(\widetilde{Q}X-\widetilde{P}Y)+2r\mathcal{R}'\varepsilon(\widetilde{q}X-\widetilde{p}Y)}\Bigg{|}_{x=g_1(X,Y),y=g_2(X,Y)},
\end{equation}
where $\mu=\mu(x,y)$ is the integrating factor of system \eqref{eq:xyp}$|_{\varepsilon=0}$ corresponding to the first integral $H$, $\mathcal{R}'$ is the derivative of $\mathcal{R}(r^2)$ with respect to $r^2$, $\widetilde{p}=\frac{\partial h_1}{\partial x}p+\frac{\partial h_1}{\partial y} q$, $\widetilde{q}=\frac{\partial h_2}{\partial x} p+\frac{\partial h_2}{\partial y} q$, and $X=\rho(r,\theta)\cos\theta$ and $Y=\rho(r,\theta)\sin\theta$.
\end{theorem}
\begin{proof}
The proof is similar to that of Theorem \ref{th:BL} (see \cite{BL} for reference), and thus we omit it here.
\end{proof}
\begin{remark}
(i) Theorem \ref{th:BL} is a special case of Theorem \ref{th:G} with $h_1(x,y)=x$, $h_2(x,y)=y$ and $\mathcal{R}(r^2)=r^2$.
(ii) In Theorem \ref{th:G}, we can also choose $X=\rho(r,\theta)\sin\theta$ and $Y=\rho(r,\theta)\cos\theta$. Moreover, in the present paper, we use this transformation for convenient calculation.
\end{remark}
For example, system $S_3$ has a first integral
\[
H_3=\dfrac{9(x^2+y^2)-24x^2y+16x^4}{3-16y}.
\]
Choose $X=3x=\rho(r,\theta)\sin\theta, Y=4x^2-3y=\rho(r,\theta)\cos\theta$, and $\mathcal{R}(r^2)=27r^2/(64(1-r^2))$,
%$H_3$ is translated to
%\begin{equation}
%\widetilde{H}_3=\dfrac{27(X^2+Y^2)}{9(9+16Y)-64X^2}.
%\end{equation}
we get $\rho(r,\theta)=9r/(8(1-r\cos\theta))$.

System $S_4$ has a first integral
\[
H_4=\dfrac{9(x^2+y^2)+24y^3+16y^4}{(3+8y)^4},
\]
and the period annulus around the isochronous center corresponds to $H_4=h, h\in(0,1/256)$.
Choose
\begin{equation}\label{tr1}
X=3x,\quad Y=4y^2+3y, \quad \mathcal{R}(r^2)=\dfrac{r^2}{256},
\end{equation}
thus
\begin{equation}\label{tr2}
\rho(r,\theta)=\dfrac{9r}{16(1-r\cos\theta)},\quad r\in(0,1).
\end{equation}

\section{ Averaged function }\label{sec:aver}
We will get the averaged function in this section.

By the polar coordinates
\begin{equation}\label{tr3}
x=\dfrac{3r\sin\theta}{16(1-r\cos\theta)},\quad y=\dfrac{3}{8}\left(-1+\dfrac{1}{\sqrt{1-r\cos\theta}}\right),\quad r\in(0,1),
\end{equation}
which can be obtained from \eqref{tr1} and \eqref{tr2}, system \eqref{S4} is transformed to the following form:
\begin{equation}\label{T1}
\dfrac{\mathrm{d}r}{\mathrm{d}\theta}=\left\{\begin{array}{ll}
\varepsilon X_1(\theta,r)+\varepsilon^2Y_1(\theta,r,\varepsilon), & \mbox{if $\sin\theta>0$,}\\
\varepsilon X_2(\theta,r)+\varepsilon^2Y_2(\theta,r,\varepsilon), & \mbox{if $\sin\theta<0$,}
\end{array} \right.
\end{equation}
where
\begin{equation}\begin{split}
X_i(\theta,r)=\dfrac{16T_i(\theta,r)}{3},\quad
Y_i(\theta,r)=-\dfrac{256T_i(\theta,r)S_i(\theta,r)}{9r+48\varepsilon S_i(\theta,r)},
\end{split}\end{equation}
with
\begin{equation}
\begin{split}
T_i(\theta,r)&=-(1-r\cos\theta)\sin\theta P_i(\theta,r)+
\sqrt{1-r\cos\theta}(r-\cos\theta) Q_i(\theta,r),\\
S_i(\theta,r)&=-(1-r\cos\theta)\cos\theta P_i(\theta,r)
+\sqrt{1-r\cos\theta}\sin\theta Q_i(\theta,r).
\end{split}
\end{equation}
Here $P_i(\theta,r)$ and $Q_i(\theta,r)$ derive from $P_i(x,y)$ and $Q_i(x,y)$ given in \eqref{PQ} by the variable translations \eqref{tr3}, $i=1,2$.

Let
\begin{equation*}
\begin{split}
F_i(\theta,r)&=\dfrac{1}{2}\left(X_1(\theta,r)-(-1)^{i}X_2(\theta,r)\right),\\
R_i(\theta,r,\varepsilon)&=\dfrac{1}{2}\left(Y_1(\theta,r,\varepsilon)-(-1)^{i}Y_2(\theta,r,\varepsilon)\right), \quad i=1,2.
\end{split}
\end{equation*}
System \eqref{T1} can be reduced to the standard form
\begin{equation}\label{T2}
\dfrac{\mathrm{d}r}{\mathrm{d}\theta}=\varepsilon F(\theta,r)+\varepsilon^2R(\theta,r,\varepsilon),
\end{equation}
where
\begin{equation*}
\begin{split}
F(\theta,r)&=F_1(\theta,r)+\mathrm{sign}(\sin\theta)F_2(\theta,r),\\
R(\theta,r,\varepsilon)&=R_1(\theta,r,\varepsilon)+\mathrm{sign}(\sin\theta)R_2(\theta,r,\varepsilon).
\end{split}
\end{equation*}
It is easy to verify that equation \eqref{T2} satisfies the three conditions of Theorem \ref{th:average}.
Thus, the averaged function is given by
\begin{equation}
\begin{split}
f(r)&=\displaystyle\int_0^{2\pi}F(\theta,r)\mathrm{d}\theta\\
&=\displaystyle\int_0^{\pi}X_1(\theta,r)\mathrm{d}\theta
+\displaystyle\int_\pi^{2\pi}X_2(\theta,r)\mathrm{d}\theta.
\end{split}
\end{equation}
Use a change of variable $\theta\rightarrow2\pi-\theta$ for the second part, hence
\begin{equation}
\begin{split}
f(r)&=\dfrac{16}{3}\displaystyle\int_0^{\pi}T_1(\theta,r)\mathrm{d}\theta
+\dfrac{16}{3}\displaystyle\int_\pi^{2\pi}T_2(\theta,r)\mathrm{d}\theta\\
&=\dfrac{16}{3}\displaystyle\int_0^{\pi}[-(1-r\cos\theta)\sin\theta P(\theta,r)+
\sqrt{1-r\cos\theta}(r-\cos\theta)Q(\theta,r)]\mathrm{d}\theta,
\end{split}
\end{equation}
where
\begin{equation*}\begin{split}
P(\theta,r)&=\displaystyle\sum_{i+j=1}^2(a_{ij}-(-1)^ic_{ij})\left(\dfrac{3r\sin\theta}{16(1-r\cos\theta)}\right)^i
\left(\frac{3}{8}\left(-1+\frac{1}{\sqrt{1-r\cos\theta}}\right)\right)^j,\\
Q(\theta,r)&=\displaystyle\sum_{i+j=1}^2(b_{ij}+(-1)^id_{ij})\left(\dfrac{3r\sin\theta}{16(1-r\cos\theta)}\right)^i
\left(\frac{3}{8}\left(-1+\frac{1}{\sqrt{1-r\cos\theta}}\right)\right)^j.
\end{split}\end{equation*}

By direct computation, we get the averaged function
\begin{equation}\label{fr}
f(r)=k_1f_1+k_2f_2+k_3f_3+k_4f_4+k_5f_5+k_6f_6,\quad  r\in(0,1),
\end{equation}
where
\begin{equation}\begin{split}
f_1&=r,\\
f_2&=\dfrac{3r+(1-r)^{\frac{3}{2}}-(1+r)^{\frac{3}{2}}}{r},\\
f_3&=-2+(1-r)^{\frac{3}{2}}+(1+r)^{\frac{3}{2}},\\
f_4&=\dfrac{2r-(1-r^2)\ln\left(\frac{1+r}{1-r}\right)}{r},\\
f_5&=rI(r),\\
f_6&=\frac{I(r)-(1-r^2)J(r)}{r},
\end{split}\end{equation}
and
\begin{equation}\begin{split}\label{k}
k_1&=\dfrac{\pi}{16}(-8(a_{10}+c_{10})+3(a_{11}+c_{11})+32(b_{01}+d_{01})-24(b_{02}+d_{02})),\\
k_2&=\dfrac{1}{6}(8(a_{01}-c_{01})-6(a_{02}-c_{02})+8(b_{10}-d_{10})-3(b_{11}-d_{11})),\\
k_3&=\dfrac{1}{4}(8(b_{10}-d_{10})-3(b_{11}-d_{11})),\\
k_4&=-\dfrac{3}{16}(a_{20}-c_{20}-2(b_{11}-d_{11})),\\
k_5&=-\dfrac{1}{8}(16(b_{01}+d_{01})+3(b_{20}+d_{20})-6(b_{02}+d_{02})),\\
k_6&=-\dfrac{1}{12}(3(a_{11}+c_{11})+8(b_{01}+d_{01})-6(b_{20}+d_{20})-12(b_{02}+d_{02})),
\end{split}\end{equation}
with
\begin{equation}\begin{split}\label{IJ}
I(r)&=\displaystyle\int_0^{\pi}\sqrt{1-r\cos\theta}\mathrm{d}\theta
=2\sqrt{1+r}E\left(\frac{2r}{1+r}\right),\\
J(r)&=\displaystyle\int_0^{\pi}\frac{1}{\sqrt{1-r\cos\theta}}\mathrm{d}\theta
=\frac{2K\left(\frac{2r}{1+r}\right)}{\sqrt{1+r}}.
\end{split}\end{equation}
Here $K(r)$ and $E(r)$ are the first and second complete elliptic integrals, respectively.

It follows from
\begin{equation}
\dfrac{\partial(k_1,k_2,k_3,k_4,k_5,k_6)}
{\partial(a_{10},a_{01},a_{20},a_{11},b_{10},b_{01})}=-\dfrac{\pi}{8}\neq0
\end{equation}
that $k_1,k_2,\cdots,k_6$ are independent.
In order to identify that $f_1,f_2,\cdots,f_6$ are linearly independent functions,
we carry out Taylor expansions in the variable $r$ around $r=0$ for these functions:
\begin{equation}\begin{split}
f_1&=r,\\
f_2&=\dfrac{r^2}{8}+\dfrac{3r^4}{128}+\dfrac{9r^6}{1024}+O(r^7),\\
f_3&=\dfrac{3r^2}{4}+\dfrac{3r^4}{64}+\dfrac{7r^6}{512}+O(r^7),\\
f_4&=\dfrac{4r^2}{3}+\dfrac{4r^4}{15}+\dfrac{4r^6}{35}+O(r^7),\\
f_5&=\pi r-\dfrac{\pi r^3}{16}-\dfrac{15\pi r^5}{1024}+O(r^7),\\
f_6&=\dfrac{3\pi r}{4}+\dfrac{9\pi r^3}{128}+\dfrac{105\pi r^5}{4096}+O(r^7).
\end{split}\end{equation}
Since the determinant of the coefficient matrix of the variables $r,r^2,r^3,r^4,r^5,r^6$ is equal to $-685\pi^2/7516192768$, these functions are linearly independent.

\section{Proof of the main result}\label{sec:proof}
This section is devoted to the study of the number of zeros of $f(r)$ obtained in \eqref{fr}. By Lemma \ref{le:CGP} below, it is easy to know that $f(r)$ has at least 5 zeros in $r\in(0,1)$. However, it is difficult to estimate the sharp upper bound of the number of zeros. We will consider $F(r)=rf(r)$ instead of $f(r)$ for convenience, which has the same zeros as $f(r)$ in $r\in(0,1)$.

\begin{lemma}\label{le:CGP}\cite{CGP}
Consider $n$ linearly independent analytical functions $f_i(x):D\rightarrow\mathbb{R},i=1,2,\cdots,n$,
where $D\subset\mathbb{R}$ is an interval. Suppose that there exists $k\in{1,2,\cdots,n}$ such that $f_k(x)$ has constant sign. Then there exists $n$ constants $c_i, i=1,2,\cdots,n$ such that $c_1f_1(x)+
c_2f_2(x)+\cdots+c_nf_n(x)$ has at least $n-1$ simple zeros in $D$.
\end{lemma}

Since the expression of $F(r)$ not only includes various elementary functions, such as $(1\pm r)^{3/2}$, $\ln [(1+r)/(1-r)]$, but also contains the first and second complete elliptic integrals, it is challenging to obtain the exact upper bound of the number of zeros of $F(r)$.

The first and the most important step is to eliminate the logarithm function. Enlightened by the idea in \cite{GI}, we get
\begin{equation}
\left(\dfrac{F(r)}{1-r^2}\right)'=\dfrac{G(r)}{(1-r^2)^2},
\end{equation}
where
\begin{equation}\label{G}
G(r)=m_1 g_1+m_2 g_2+m_3 g_3+m_4 g_4+m_5 g_5+m_6 g_6,
\end{equation}
with
\begin{equation}\begin{split}\label{g}
g_1&=r,\\
g_2&=r^2,\\
g_3&=6+4r(\sqrt{1-r}-\sqrt{1+r})-(3+r^2)(\sqrt{1-r}+\sqrt{1+r}),\\
g_4&=-4+2(1+r^2)(\sqrt{1-r}+\sqrt{1+r})+r(-5+r^2)(\sqrt{1-r}-\sqrt{1+r}),\\
g_5&=r\left((-5+r^2)I(r)+(1-r^2)J(r)\right),\\
g_6&=r\left(4I(r)-(1-r^2)J(r)\right),
\end{split}\end{equation}
and
\begin{equation}\begin{split}
m_1=2k_1,\,\,\, m_2=3k_2-2k_3+4k_4,\,\,\, m_3=\dfrac{k_2}{2},\,\,\,m_4=\dfrac{k_3}{2},\,\,\, m_5=-\dfrac{k_5}{2},\,\,\, m_6=\dfrac{k_6}{2}.
\end{split}\end{equation}
Here $I(r), J(r)$ are defined in \eqref{IJ}, and $k_i, i=1,2,\cdots,6$ are given in \eqref{k}. The independence of $m_i, i=1,2,\cdots,6$ come from the independence of $k_i, i=1,2,\cdots,6$.
Thus,
\begin{equation}\label{F}
F(r)=(1-r^2)\int_0^r(1-\xi^2)^{-2}G(\xi)\mathrm{d}\xi,
\end{equation}
and $F(r)$ has at most as many zeroes as $G(r)$ in $(0,1)$.

In the following, we will use \emph{Chebyshev
criterion} to show that $G(r)$ has at most 5 zeroes. We introduce some definitions and lemmas first, see for instance \cite{MV}.

Let $g_{1},g_{2},...,g_{n}$ be analytic functions on an open
interval $L$ of $\mathbb{R}$. An ordered set $(g_{1},g_{2},...,g_{n})$ is an \emph{extended
complete Chebyshev system} (in short, ECT-system) on $L$ if, for all
$i=1,2,...,n$, any nontrivial linear combination
\begin{equation}\label{ECT}
\lambda_{1}g_{1}(x)+\lambda_{2}g_{2}(x)+...+\lambda_{i}g_{i}(x)
\end{equation}
has at most $i-1$ isolated zeros on $L$ counted with multiplicities.

\begin{remark}\label{ECT2}
If $(g_{1},g_{2},...,g_{n})$ is an ECT-system on $L$, then for
each $i=1,2,...,n$, there exists a linear combination \eqref{ECT} with
exactly $i-1$ simple zeros on $L$ (see for instance Remark 3.7 in \cite{GV}).
\end{remark}

\begin{lemma}\label{ECT3}
\cite{MV} $(g_{1},g_{2},...,g_{n})$ is an ECT-system on L
if, and only if, for each $i=1,2,...,n$,
\begin{equation*}
W_{i}(x)=
\begin{vmatrix}
g_{1}(x) &  g_{2}(x) & \cdots &  g_{i}(x)\\
g_{1}^{\prime}(x) &g_{2}^{\prime}(x)  & \cdots & g_{i}^{\prime}(x) \\
\vdots & \vdots & \ddots & \vdots \\
g_{1}^{(i-1)}(x) & g_{2}^{(i-1)}(x) & \cdots & g_{i}^{(i-1)}(x) \\
  \end{vmatrix}\neq{0}, \quad \mbox{$x\in L$} .
\end{equation*}
\end{lemma}
The following is to prove that $(g_1, g_2, \cdots, g_6)$ in \eqref{g} is an ECT-system. A direct calculation leads to

\begin{equation}\begin{split}\label{W1W6}
W_1(r)=&r,\\
W_2(r)=&r^2,\\
W_3(r)=&\frac{3(\sqrt{1+r}(-8+4r+r^2+r^3)-\sqrt{1-r}(8+4r-r^2+r^3)+16\sqrt{1-r^2})}{4\sqrt{1-r^2}},\\
W_4(r)=&\frac{15}{8(1-r^2)^{\frac{3}{2}}}\left(r(80-81r^2+9 r^4)+4(1+r)^{\frac{3}{2}}(6-29r+21r^2)\right.\\
&\left.-4(1-r)^{\frac{3}{2}}(6+29r+21r^2)-r\sqrt{1-r^2}(-80+41r^2+3r^4)\right),\\
W_5(r)=&\frac{15}{64r^3(1-r^2)^{\frac{7}{2}}}\left(W_{51}(r)I(r)+W_{52}(r)J(r)\right),\\
W_6(r)=&\frac{675}{1024r^5(1-r^2)^{6}}\left(W_{61}(r)I^2(r)+W_{62}(r)I(r)J(r)+W_{63}(r)J^2(r)\right),\\
\end{split}\end{equation}
where
\begin{equation*}\begin{split}
W_{51}(r)=&-r(3840-3760r^2+2082r^4-2649r^6+351r^8)\\
       &+2(1+r)^{\frac{3}{2}}(-384+2496r-744r^2-1724r^3+702r^4+573r^5-1512r^6+525r^7)\\
       &+2(1- r)^{\frac{3}{2}}(384+2496r+744r^2-1724r^3-702r^4+573r^5+1512r^6+525r^7)\\
       &-2r\sqrt{1-r^2}(1920-920r^2+917r^4-1125r^6+36 r^8),\\
W_{52}(r)=&\sqrt{1-r^2}\left(-2r(-1920+2360r^2-1243r^4+164r^6+27r^8)\right.\\
       &+2\sqrt{1-r}(1+r)^2(384-2496r+1032r^2+812r^3-1098r^4+735r^5-315r^6)\\
       &+2\sqrt{1+r}(1-r)^2(-384-2496r-1032r^2+812r^3+1098r^4+735r^5+315r^6)\\
       &\left.-r\sqrt{1-r^2}(-3840+2800r^2-1374r^4-729r^6+27r^8)\right),\\
W_{61}(r)=&r(21760-42320r^2+7456r^4+25161r^6-13566r^8+1845r^{10})\\
       &-2\sqrt{1-r}(1+r)^2(768+9728r-1808r^2-16608r^3-2904r^4+7416r^5+6531r^6\\
       &-1764r^7-945r^8+630r^9)-2\sqrt{1+r}(1-r)^2(-768+9728r+1808r^2\\
       &-16608r^3+2904r^4+7416r^5-6531r^6-1764r^7+945r^8+630 r^9)\\
       &+5r\sqrt{1-r^2}(4352-6288r^2-1160r^4+4143r^6-549r^8+54r^{10}),\\
W_{62}(r)=&2\sqrt{1-r^2}\left(5r(-2816+3632r^2+992r^4-1391r^6-492r^8+27r^{10})\right.\\
       &+2(1-r)^{\frac{5}{2}}(-768+5120r+6160r^2-3520r^3-744r^4+4440r^5\\
       &+2277r^6+630r^7-315r^8)+2(1+r)^{\frac{5}{2}}(768+5120r-6160r^2\\
       &-3520r^3+744r^4+4440r^5-2277r^6+630r^7+315r^8)\\
      \end{split}\end{equation*}       
\begin{equation*}\begin{split}
 &\left.-r\sqrt{1-r^2}(14080-11120r^2-8504r^4+1233r^6+570r^8+45r^{10})\right),\\       
W_{63}(r)=&(1-r^2)\left(r(6400+2160r^2-17904r^4+10049r^6-2910r^8+45r^{10})\right.\\
       &+2\sqrt{1-r}(1+r)^3(-768-1280r+784r^2-2000r^3+4584r^4-720r^5\\
\end{split}\end{equation*}
\begin{equation*}\begin{split}
&-1155r^6+525r^7)+2\sqrt{1+r}(1-r)^3(768-1280r-784r^2-2000r^3\\
       &-4584r^4-720r^5+1155r^6+525r^7)\\
       &\left.-5r\sqrt{1-r^2}(-1280-1072r^2+2936r^4-935r^6+309r^8+18r^{10})\right),
\end{split}\end{equation*}
and $I(r)$, $J(r)$ are complete elliptic integrals defined in \eqref{IJ}.

Taking account of the complexity of $W_3$, $W_4$, $W_5$ and $W_6$, we do a transformation
\begin{equation}\label{rs}
r=\dfrac{-1+6s^2-s^4}{(1 + s^2)^2},\quad \quad s\in(\sqrt{2}-1,1),
\end{equation}
and simplify them to
\begin{equation}\begin{split}\label{W3W6}
W_3(s)=&\frac{3(s-\sqrt{2}+1)^4}{8 \left(s-s^3\right)\left(1+s^2\right)^5}\left(-24-17\sqrt{2}+8(3+2\sqrt{2})s+(80+57\sqrt{2})s^2+8(2+\sqrt{2}) s^3\right.\\
&-18(16+11\sqrt{2})s^4-80 (4 + 3 \sqrt{2}) s^5-18(8+5\sqrt{2})s^6-8 (10 + 7 \sqrt{2}) s^7\\
&\left. - (8+9 \sqrt{2}) s^8 -8 (3 + 2 \sqrt{2}) s^9 + \sqrt{2} s^{10}\right),\\
W_4(s)=&\frac{15(s+\sqrt{2}+1) (s -\sqrt{2}+1)^7}{
 64 (-s + s^3)^3 (1 + s^2)^6} \left(181 + 128 \sqrt{2} +
   2 (601 + 425 \sqrt{2}) s + 6 (383 \right.\\
   &+ 271 \sqrt{2}) s^2+
   2 (1433 + 1015 \sqrt{2}) s^3 + 6 (1006 + 711 \sqrt{2}) s^4 +
   6 (2087 + 1477 \sqrt{2}) s^5\\
   &+ 2 (8267 + 5839 \sqrt{2}) s^6 +
   22 (439 + 311 \sqrt{2}) s^7 + 54 (7 + 5 \sqrt{2}) s^8 -
   22 (79 + 59 \sqrt{2}) s^9\\
   &  - 2 (973 + 629 \sqrt{2}) s^{10} -
   6 (167 + 133 \sqrt{2}) s^{11} - 6 (54 + 31 \sqrt{2}) s^{12}\\
   &\left. -
   2 (233 + 175 \sqrt{2}) s^{13} + 6 (23 + 19 \sqrt{2}) s^{14} -
   2 (1 + 5 \sqrt{2}) s^{15} + (1 + 2 \sqrt{2}) s^{16}\right),\\
W_5(s)=&Z_5(s)\left(Z_{51}(s)I\begin{small}\left(r\right)\end{small}
  +Z_{52}(s)J\left(r\right)\right)|_{r=(-1+6s^2-s^4)/(1 + s^2)^2},\\
W_6(s)=&Z_6(s)\left(Z_{61}(s)I^2(r)+Z_{62}(s)I(r)J(r)+Z_{63}(s)J^2(r)\right)|_{r=(-1+6s^2-s^4)/(1 + s^2)^2},
\end{split}\end{equation}
where $Z_5(s)$ and $Z_6(s)$ are rational functions, and $Z_{51}(s)$, $Z_{52}(s)$, $Z_{61}(s)$, $Z_{62}(s)$ and $Z_{63}(s)$ are polynomials. Since their expressions are cumbersome,  we place them in Appendix given by \eqref{W5W6}. Under the transformation \eqref{rs}, $W_i(s)=W_i(r),i=3,4,5,6$.

It is easy to judge that neither $W_3(s)$ nor $W_4(s)$ vanishes in $s\in(\sqrt{2}-1,1)$ by Sturm's Theorem. The difficulties mainly focus on the determination of $W_5(s)$ and $W_6(s)$.

From the definition of $I(r)$ in \eqref{IJ} and the transformation \eqref{rs}, we know that
\begin{equation}\label{I}
I(r)|_{r=(-1+6s^2-s^4)/(1 + s^2)^2}>0.
\end{equation}
Define
\begin{equation}\label{v}
v(s)=\dfrac{J(r)}{I(r)}\bigg{|}_{r=(-1+6s^2-s^4)/(1 + s^2)^2}, \quad s\in(\sqrt{2}-1,1).
\end{equation}
\begin{proposition}\label{W5}
$W_5(s)$ does not vanish in $s\in(\sqrt{2}-1,1)$.
\end{proposition}
\begin{proof}
It follows from Sturm's Theorem that $Z_{52}(s)<0$ in $(\sqrt{2}-1,1)$. Combining \eqref{I},
\[
W_5(s)=Z_5(s)Z_{52}(s)I\left(\frac{-1+6s^2-s^4}{(1 + s^2)^2}\right)\left(\frac{Z_{51}(s)}{Z_{52}(s)}+v(s)\right).
\]
It suffices to consider
\[
U(s)=\dfrac{Z_{51}(s)}{Z_{52}(s)}+v(s).
\]
Using
\begin{equation}\label{v/s}
\dfrac{\mathrm{d}v}{\mathrm{d}s}=\dfrac{(1 + s^2)^4 - 32 s^2 (1-s^2)^2v +16 s^2 (1-s^2)^2v^2}{2 s (-1 + s^4) (1 - 6 s^2 + s^4)},
\end{equation}
which can be obtained from \eqref{IJ} and \eqref{v} by direct computations, we have
\begin{equation}\begin{split}\label{dU}
\dfrac{\mathrm{d}U}{\mathrm{d}s}=&-\frac{8 (s-s^3) U^2}{(1 + s^2) (1 - 6 s^2 + s^4)}+\frac{
 16 (s-s^3)(Z_{51}(s) +Z_{52}(s))U}{(1 + s^2) ( 1 - 6 s^2 + s^4 ) Z_{52}(s)}\\
 &-\frac{36(1-2 s-s^2) (1 + s^2)^3 (1 + 2 s - s^2)^3 \bar{Z}(s)}{Z_{52}^2(s)},
\end{split}\end{equation}
where $\bar{Z}(s)$ is a polynomial of degree 56 given by \eqref{Ws} in Appendix, and
$\bar{Z}(s)<0$ in $(\sqrt{2}-1,1)$ by Sturm Theorem.

If $U(s)$ has a zero $s_0$ in $(\sqrt{2}-1,1)$, then from \eqref{dU} and $\bar{Z}(s)<0$, we know that $U'(s_0)<0$. However, the Taylor expansion of $U(s)$ near $s=\sqrt{2}-1$ is
\begin{equation}\begin{split}\label{s0}
U(s)=&-\dfrac{1}{5}\left(3+2\sqrt2\right)\left(s-(\sqrt2-1)\right)^2+O\left((s-(\sqrt2-1))^3\right),
\end{split}\end{equation}
and when $s\rightarrow1^-$,
\begin{equation}\begin{split}\label{s1}
U(s)=&\dfrac{1- 2 \sqrt{2}}{18(1-s)}-\dfrac{1}{2}\ln(1-s)+\cdots.
\end{split}\end{equation}
If $U(s)$ has zeros in $(\sqrt2-1,1)$, then it has at least two (taking into account their
multiplicity), and one of them satisfies $U'(s_0)>0$ or $U'(s_0)=0$ by \eqref{s0} and \eqref{s1}, which contradicts with $U'(s_0)<0$. Thus, we have $U(s)$ has none zeros in $(\sqrt2-1,1)$, i.e., $W_{5}(s)$ does not vanish in $(\sqrt2-1,1)$.
\end{proof}
\begin{proposition}\label{W6}
$W_6(s)$ does not vanish in $s\in(\sqrt{2}-1,1)$.
\end{proposition}
\begin{proof}
Since $I(r)|_{r=(-1+6s^2-s^4)/(1 + s^2)^2}>0$, $W_6(s)$ can be expressed as
\[
W_6(s)=Z_6(s)I^2\left(\frac{-1+6s^2-s^4}{(1 + s^2)^2}\right)\left(Z_{61}(s)
+Z_{62}(s)v(s)+Z_{63}(s)v^2(s)\right).
\]
Let
\begin{equation}
\Psi(s,v)=Z_{61}(s)+Z_{62}(s)v+Z_{63}(s)v^2.
\end{equation}
Then the number of zeros of $W_6(s)$ in $(\sqrt{2}-1,1)$ equals the number of intersection points of the
curve $C=\{\Psi(s,v)=0\}$ and the curve $\Gamma=\{v=v(s)\}$
in the $(s,v)$-plane. Let $C_{+}$ and $C_{-}$ be two branches of the curve $C$, denoted by
\begin{equation}
C_{+}=\{v_{+}(s)=\dfrac{-Z_{62}+\sqrt{\Delta(s)}}{2Z_{63}}\}, \quad
C_{-}=\{v_{-}(s)=\dfrac{-Z_{62}-\sqrt{\Delta(s)}}{2Z_{63}}\},
\end{equation}
where $Z_{63}(s)<0$ and $\Delta(s)=Z_{62}^2-4Z_{61}Z_{63}>0$ in $s\in(\sqrt2-1,1)$ by Sturm Theorem, which are polynomials given
by \eqref{W5W6} and \eqref{delta}, respectively, in Appendix.

On one hand, note that
\begin{equation*}\begin{split}
\lim_{s\rightarrow(\sqrt2-1)^+}v_{-}(s)=13, \quad \lim_{s\rightarrow1^-}v_{+}(s)=-\infty,
\end{split}\end{equation*}
near $s=\sqrt2-1$,
\begin{equation*}\begin{split}
v_{+}(s)=&1+\frac{3 + 2\sqrt{2}}{2}\left(s-(\sqrt2-1)\right)^2-\frac{4 + 3\sqrt{2}}{4}\left(s-(\sqrt2-1)\right)^3\\ &+\frac{13+9\sqrt{2}}{8}\left(s-(\sqrt2-1)\right)^4+O\left((s-(\sqrt2-1))^5\right),\\
v(s)=&1+\frac{3 + 2\sqrt{2}}{2}\left(s-(\sqrt2-1)\right)^2-\frac{4 + 3\sqrt{2}}{4}\left(s-(\sqrt2-1)\right)^3\\ &+\frac{103+72\sqrt{2}}{32}\left(s-(\sqrt2-1)\right)^4+O\left((s-(\sqrt2-1))^5\right),
\end{split}\end{equation*}
and when $s\rightarrow1^-$,
\begin{equation*}\begin{split}
v_{-}(s)&=\dfrac{(181 + 128\sqrt{2})(35+
   2 \sqrt{35 (29 + 9\sqrt{2})})(\sqrt{2}-1)^6 }{630(1-s)}+\cdots,\\
v(s)&=-\dfrac{1}{2}\ln(1-s)+\cdots.
\end{split}\end{equation*}
Comparing these results, we have
\begin{equation}\begin{split}\label{0}
&v_{-}(s)>v(s)>v_{+}(s),\quad \mbox{near $s=\sqrt2-1$},\\
&v_{-}(s)>v(s)>v_{+}(s), \quad \mbox{near $s=1$}.
\end{split}\end{equation}
On the other hand, from \eqref{v/s} we know that $v(s)$ defined in \eqref{v} is a solution of
\begin{equation}
\begin{split}\label{vs}
\dot{v}&=(1 + s^2)^4 - 32 s^2 (1-s^2)^2v +16 s^2 (1-s^2)^2v^2,\\
\dot{s}&=2 s (-1 + s^4) (1 - 6 s^2 + s^4).
\end{split}
\end{equation}
A direct calculation shows that
\begin{equation}\begin{split}
\Phi(s,v)=\left(\frac{\partial \Psi(s,v)}{\partial s},\frac{\partial \Psi(s,v)}{\partial v}\right)\cdot(\dot{s},\dot{v})=4 s (1 - s^2)\sum_{k=0}^{k=3}\phi_k(s)v^k,
\end{split}\end{equation}
where $\phi_k(s),k=0,1,2,3$ are polynomials given by \eqref{phi} in Appendix.

It follows from Sturm's Theorem that $\phi_3(s)<0$ in $(\sqrt2-1,1)$, and that the resultant $R=\mbox{Resultant}(\Psi,\Phi,v)$ of $\Psi(s,v)$ and $\Phi(s,v)$ with respect to $v$, which is a polynomial in the variable $s$ of degree 250 given by \eqref{R} in Appendix, has a unique simple zero $s_0$ in $(\sqrt2-1,1)$. Thus $v_{-}(s)>v(s)>v_{+}(s)$ holds for all $s\in(\sqrt2-1,1)$ by \eqref{0}, otherwise, there will exist at least two points on the curve $C$ tangent to the vector field \eqref{vs}, which results in a contradiction. We have $W_6(s)$ does not vanish in $s\in(\sqrt2-1,1)$.
\end{proof}
\noindent{\bfseries{Proof of Theorem \ref{th:S4}.}} Using \eqref{W1W6}, \eqref{W3W6}, Propositions \ref{W5} and \ref{W6}, we have $W_1(r), W_2(r)$, $r\in(0,1)$ and $W_3(s)$, $W_4(s), W_5(s)$, $W_6(s), s\in(\sqrt2-1,1)$ do not vanish, where $W_i(s), s\in(\sqrt2-1,1)$ are equal to $W_i(r), r\in(0,1)$, $i=3,4,5,6$, under the transformation \eqref{rs}. Thus $(g_1, g_2, \cdots, g_6)$ defined in \eqref{g} is an ECT-system by Lemma \ref{ECT3}, which implies that $G(r)$ given in \eqref{G} has at most 5 zeros in $(0,1)$, and this number is realizable
by Remark \ref{ECT2}. Note that $f(r)=F(r)/r$ in $(0,1)$ and $F(r)$ has a relation with $G(r)$ in \eqref{F}, we know that the
averaged function $f(r)$ given by \eqref{fr} has at most 5 zeros in $(0,1)$. Theorem \ref{th:S4} follows from Theorem \ref{th:average}.
%\section{Example}
%Consider
%\begin{equation*}
%\begin{split}
%a_{10}&=5.21662, \,\,\, \,\,\,\quad a_{01}=37.8953, \quad \,\,\,\,\,\, \,a_{20}=33.3637,\\
%a_{11}&=-26.0859, \quad\,\, b_{10}=1.19618, \quad\quad\, b_{01}=-5.21779,
%\end{split}
%\end{equation*}
%\[
%a_{02}=b_{11}=b_{20}=b_{02}=c_{10}=c_{01}=c_{20}=c_{11}=c_{02}=d_{10}=d_{01}=d_{20}=d_{11}=d_{02}=0.
%\]
%such that
%\begin{equation*}
%\begin{split}
%a_{1}&=-56.3445,\quad a_2=52.122,\quad \,\,\, a_3=2.39237,\\
%a_4&=-6.25569,\quad a_5=10.4356,\quad a_6=10.
%\end{split}
%\end{equation*}
%the zeros are
%\[
%r_1=\dfrac{1}{10},\quad r_2=\dfrac{3}{10}, \quad r_3=\dfrac{1}{2}, \quad r_4=\dfrac{7}{10},\quad r_5=\dfrac{9}{10}.
%\]}
\section{Acknowledgements}
We appreciate the helpful suggestion of Professor Changjian Liu. The research is supported by the NSF of China (No. 11401111, No. 11571379 and No. 11571195).

\vspace{250pt}
\section*{Appendix}\label{sec:app}
\vspace{-20pt}
\begin{equation*}\begin{split}
Z_5(s)=&\frac{15(s -\sqrt{2}+1)^4}{131072 (s + \sqrt{2} + 1)^2 (1 + s^2)^2(-1 - 2 s + s^2)^3 (s - s^3)^7},\\
Z_{51}(s)=&(1 + s^2)^2 \left(17 (181 + 128 \sqrt{2})+10(13270+9383\sqrt{2})s+2 (458575 + 324203 \sqrt{2}) s^2\right.\\
&+6 (573334 + 405377 \sqrt{2}) s^3 +2 (5125412 + 3624289 \sqrt{2}) s^4 +2 (10738190\\
& + 7593409 \sqrt{2}) s^5 +2 (14254381 + 10078037 \sqrt{2}) s^6 +18 (2125642 + 1501057 \sqrt{2}) s^7\\
& +6 (19300116 + 13648361 \sqrt{2}) s^8 +2 (129946918 + 91939979 \sqrt{2}) s^9\\
&+6 (34318349 + 24275769 \sqrt{2}) s^{10} -6 (23220354 + 16456951 \sqrt{2}) s^{11}\\
& -78 (3894648 + 2757089 \sqrt{2}) s^{12} +2 (10979494 + 7760369 \sqrt{2}) s^{13}\\
&+2 (175579757 + 124112365 \sqrt{2}) s^{14} +6 (26735878 + 19250787 \sqrt{2}) s^{15}\\
&-38206562 (7 + 5 \sqrt{2}) s^{16} -6 (48258258 + 34316453 \sqrt{2}) s^{17}-2 (6664843\\
& + 3458855 \sqrt{2}) s^{18} +2 (518246 + 288049 \sqrt{2}) s^{19} -78 (422308 + 326451 \sqrt{2}) s^{20}\\
& +6 (5158094 + 3813369 \sqrt{2}) s^{21} +6 (1091109 + 1016701 \sqrt{2}) s^{22} -2 (6852178\\
   & + 5773661 \sqrt{2}) s^{23} +6 (59056 + 179619 \sqrt{2}) s^{24} +18 (290578 + 190297 \sqrt{2}) s^{25}\\
   & -2 (258539 + 81007 \sqrt{2}) s^{26} - 2 (-3550 + 74191 \sqrt{2}) s^{27} +2 (-15328 \\
   &+ 25771 \sqrt{2}) s^{28} + 6 (7286 + 1057 \sqrt{2}) s^{29} -2 (10505 + 4153 \sqrt{2}) s^{30} - 10 (-110\\
   & \left.+ 17 \sqrt{2}) s^{31} +17 (1 + 2 \sqrt{2}) s^{32}\right),\\
Z_{52}(s)=&4 (1 - s) s (1 + s)\left (-153 (99 + 70 \sqrt{2}) -2 (69529 + 49157 \sqrt{2}) s - 2 (456914\right.\\
   & + 323083 \sqrt{2}) s^2 -2 (1965217 + 1389569 \sqrt{2}) s^3 -6 (1545650 + 1092881 \sqrt{2}) s^4\\
   & -2 (7529377 + 5323013 \sqrt{2}) s^5 -6 (4991094 + 3527921 \sqrt{2}) s^6 -2 (36612329\\
   & + 25895905 \sqrt{2}) s^7 -2 (60820588 + 43006727 \sqrt{2}) s^8 -6 (21189071\\
   & + 14974803 \sqrt{2}) s^9 -2 (75631610 + 53466727 \sqrt{2}) s^{10} -2 (99767893\\
   & + 70613285 \sqrt{2}) s^{11} -6 (20320850 + 14367253 \sqrt{2}) s^{12} + 2 (42604091\\
   & + 30195415 \sqrt{2}) s^{13} +2 (97432918 + 68663777 \sqrt{2}) s^{14} +2 (26485651\\
   & + 19489163 \sqrt{2}) s^{15} -27681242 (7 + 5 \sqrt{2}) s^{16} -2 (106403371 + 75431567 \sqrt{2}) s^{17}\\
   & -2 (32930102 + 22590337 \sqrt{2}) s^{18} - 2 (9553091 + 7059715 \sqrt{2}) s^{19} +6 (348730\\
   & + 101453 \sqrt{2}) s^{20} +2 (8838493 + 6962705 \sqrt{2}) s^{21} +2 (2187610 + 1006727 \sqrt{2}) s^{22}\\
   & - 6 (1245609 + 729473 \sqrt{2}) s^{23} -2 (-296432 + 224813 \sqrt{2}) s^{24} +2 (806129\\
   & + 831565 \sqrt{2}) s^{25} +6 (209366 + 112401 \sqrt{2}) s^{26} -2 (186503 + 78103 \sqrt{2}) s^{27}\\
    &+ 6 (16010 + 281 \sqrt{2}) s^{28} +2 (-16823 + 2141 \sqrt{2}) s^{29} - 2 (-2866 + 1237 \sqrt{2}) s^{30}\\
    &\left. -2 (1391 + 487 \sqrt{2}) s^{31} + 153 s^{32}\right),
\end{split}\end{equation*}
\begin{equation*}\begin{split}
Z_6(s)=&\frac{675 (s -\sqrt{2}+1)^2(1 + s^2)^6}{1073741824(s + \sqrt{2} + 1)^4(1+2s-s^2)^5(s-s^3)^{12}},\\
Z_{61}(s)=&(1 + s^2)^4 \left(21 (99 + 70 \sqrt{2}) - 6 (7263 + 5137 \sqrt{2}) s -18 (68094 + 48149 \sqrt{2}) s^2\right.\\
    &-2 (5905159 + 4175447 \sqrt{2}) s^3 - 2 (21587809 + 15265427 \sqrt{2}) s^4 -2 (36235649\\
    &+ 25634773 \sqrt{2}) s^5 -2 (38183186 + 27005059 \sqrt{2}) s^6 -2 (104249347\\
    & + 73692677 \sqrt{2}) s^7 -15 (64399407 + 45545704 \sqrt{2}) s^8 -24 (86877585\\
  & + 61428647 \sqrt{2}) s^9 -24 (38235726 + 27025957 \sqrt{2}) s^{10} +24 (28473969 \\
  &+ 20139277 \sqrt{2}) s^{11} -24 (242172657 + 171545851 \sqrt{2}) s^{12} -24 (507325001 \\
  &+ 359095513 \sqrt{2}) s^{13} +120 (16185366 + 11931665 \sqrt{2}) s^{14} +8 (2222062699 \\
   &+ 1572249221 \sqrt{2}) s^{15} +2 (4397159743 + 3033158714 \sqrt{2}) s^{16} +28 (343153879 \\
  & + 242932601 \sqrt{2}) s^{17} +4 (4783175926 + 3446081753 \sqrt{2}) s^{18} -4 (4438683541\\
  & + 3150109037 \sqrt{2}) s^{19} -6070512468 (7 + 5 \sqrt{2}) s^{20} +4 (1585594621\\
   & + 1152946793 \sqrt{2}) s^{21} +
   4 (8917028746 + 6339778727 \sqrt{2}) s^{22} -
   28 (38330119 \\&+ 29555969 \sqrt{2}) s^{23} -
   2 (10676594597 + 7518469324 \sqrt{2}) s^{24} -
   8 (130683739 \\&+ 108283949 \sqrt{2}) s^{25} +
   120 (68081866 + 48259215 \sqrt{2}) s^{26} +
   24 (48196721\\& + 37705717 \sqrt{2}) s^{27} -
   24 (41326097 + 30953259 \sqrt{2}) s^{28} -
   24 (575849 \\&+ 610593 \sqrt{2}) s^{29} +
   24 (1702894 + 931077 \sqrt{2}) s^{30} +
   24 (-870335 + 5103 \sqrt{2}) s^{31}\\& -
   15 (857267 + 1066206 \sqrt{2}) s^{32} -
   2 (3710573 + 1879267 \sqrt{2}) s^{33} -
   2 (572846\\& + 677821 \sqrt{2}) s^{34} +
   2 (1538969 + 1347097 \sqrt{2}) s^{35} -
   2 (-33311 + 130643 \sqrt{2}) s^{36}\\& +
   2 (-48161 + 8123 \sqrt{2}) s^{37} - 18 (-446 + 171 \sqrt{2}) s^{38} +
   6 (143 + 153 \sqrt{2}) s^{39}\\&\left. - 21 s^{40}\right),\\
\end{split}\end{equation*}
\begin{equation*}\begin{split}
Z_{62}(s)=&8 s(1-s^2)(1 + s^2)^2\left(15 (181 + 128 \sqrt{2}) +
   6 (22706 + 16055 \sqrt{2}) s + 6 (296263\right.\\& + 209487 \sqrt{2}) s^2 +
   6 (1214926 + 859127 \sqrt{2}) s^3 +
   2 (6723205 + 4756271 \sqrt{2}) s^4\\& +
   2 (7540306 + 5332541 \sqrt{2}) s^5 +
   2 (27655109 + 19550617 \sqrt{2}) s^6 +
   2 (133081238\\& + 94111093 \sqrt{2}) s^7 +
   9 (60326011 + 42667278 \sqrt{2}) s^8 +
   8 (45317174 \\&+ 32063287 \sqrt{2}) s^9 +
   24 (12478503 + 8804773 \sqrt{2}) s^{10} +
   24 (73177190\\& + 51754977 \sqrt{2}) s^{11} +
   24 (149240653 + 105605159 \sqrt{2}) s^{12} +
   8 (440708362\\& + 311224523 \sqrt{2}) s^{13} +
   24 (159064331 + 112637773 \sqrt{2}) s^{14} +
   8 (662194054\\& + 467790575 \sqrt{2}) s^{15} +
   2 (717532091 + 505999012 \sqrt{2}) s^{16} -
   4 (2196382646\\& + 1551096913 \sqrt{2}) s^{17} -
   4 (3614754613 + 2556199193 \sqrt{2}) s^{18} -
   12 (492996606\\ & + 351646643 \sqrt{2}) s^{19} +
   836472612 (7 + 5 \sqrt{2}) s^{20} +
   12 (423866026 \\&+ 303255237 \sqrt{2}) s^{21} -
   4 (7180333 + 30897197 \sqrt{2}) s^{22} -
   4 (288314134\\
\end{split}\end{equation*}
\begin{equation}\begin{split}\label{W5W6}
 & + 188190833 \sqrt{2}) s^{23} -
   2 (195815329 + 133344182 \sqrt{2}) s^{24} +
   8 (66530846\\ & + 42316855 \sqrt{2}) s^{25} +
   24 (21919451 + 16636357 \sqrt{2}) s^{26} +
   8 (58694618\\& + 38357563 \sqrt{2}) s^{27} +
   24 (9897613 + 8065031 \sqrt{2}) s^{28} -
   24 (1154970\\& + 1339423 \sqrt{2}) s^{29} -
   24 (2703577 + 1822683 \sqrt{2}) s^{30} -
   8 (2459954 + 2063233 \sqrt{2}) s^{31}\\& +
   9 (1143831 + 1239752 \sqrt{2}) s^{32} -
   2 (510458 + 1311547 \sqrt{2}) s^{33} -
   2 (769411\\& + 346547 \sqrt{2}) s^{34} -
   2 (65446 + 100139 \sqrt{2}) s^{35} +
   2 (280645 + 246479 \sqrt{2}) s^{36}\\& - 6 (106 + 8753 \sqrt{2}) s^{37} +
   6 (-1857 + 803 \sqrt{2}) s^{38} - 6 (-194 + 25 \sqrt{2}) s^{39}\\&\left. +
   15 (1 + 2 \sqrt{2}) s^{40}\right),\\
Z_{63}(s)=&16 s^2 (1-s^2)^2 \left(-135 (99 + 70 \sqrt{2}) -
   30 (7525 + 5321 \sqrt{2}) s - 6 (152746\right.\\& + 108015 \sqrt{2}) s^2 -
   6 (227183 + 160757 \sqrt{2}) s^3 +
   6 (432025 + 305483 \sqrt{2}) s^4\\& +
   2 (9742589 + 6888895 \sqrt{2}) s^5 +
   2 (11052254 + 7811245 \sqrt{2}) s^6 -
   2 (63240653 \\&+ 44719837 \sqrt{2}) s^7 + (-397483285 -
      281030456 \sqrt{2}) s^8 - 8 (10862657\\& + 7667785 \sqrt{2}) s^9 +
   248 (5120762 + 3620671 \sqrt{2}) s^{10} +
   24 (78243119 \\&+ 55256821 \sqrt{2}) s^{11} -
   24 (65939253 + 46640495 \sqrt{2}) s^{12} -
   8 (1174003525\\& + 829395239 \sqrt{2}) s^{13} -
   168 (72000814 + 50921581 \sqrt{2}) s^{14} -
   8 (65378599\\& + 48038387 \sqrt{2}) s^{15} +
   6 (1806045881 + 1280295926 \sqrt{2}) s^{16} +
   4 (476075279 \\&+ 336969379 \sqrt{2}) s^{17} -
   4 (2738422394 + 1941231319 \sqrt{2}) s^{18} -
   12 (307131301\\& + 218227703 \sqrt{2}) s^{19} +
   884115356 (7 + 5 \sqrt{2}) s^{20} +
   12 (145879621\\& + 105351527 \sqrt{2}) s^{21} -
   4 (668567654 + 492333001 \sqrt{2}) s^{22} -
   4 (44260439\\& + 34698991 \sqrt{2}) s^{23} +
   6 (442887421 + 326085004 \sqrt{2}) s^{24} +
   8 (252892879 \\&+ 179298383 \sqrt{2}) s^{25} -
   168 (940754 + 1179539 \sqrt{2}) s^{26} -
   8 (111015515 \\&+ 70118089 \sqrt{2}) s^{27} -
   24 (1683253 + 1661295 \sqrt{2}) s^{28} +
   24 (10113841\\& + 6593051 \sqrt{2}) s^{29} -
   248 (61498 + 6911 \sqrt{2}) s^{30} -
   8 (1913143 + 1275275 \sqrt{2}) s^{31}\\& + (6581375 +
      1814806 \sqrt{2}) s^{32} + 2 (-47467 + 418153 \sqrt{2}) s^{33} -
   2 (598846 \\&+ 344525 \sqrt{2}) s^{34} -
   2 (-71011 + 19375 \sqrt{2}) s^{35} + 6 (-2855 + 1067 \sqrt{2}) s^{36}\\& \left.+
   6 (14863 + 12133 \sqrt{2}) s^{37} - 6 (246 + 1265 \sqrt{2}) s^{38} +
   30 (-35 + 29 \sqrt{2}) s^{39} + 135 s^{40}\right).\\
\end{split}\end{equation}
\begin{equation}\label{Ws}
\begin{split}
\bar{Z}(s)=&-289(35839 + 25342 \sqrt{2})-68(625921+442593\sqrt{2})s -6 (101193745\\&
 + 71554784 \sqrt{2}) s^2-20 (639756151 + 452375917 \sqrt{2}) s^3 -6 (14433578283\\&
  + 10206081086 \sqrt{2}) s^4 - 4 (61523400599 + 43503613425 \sqrt{2}) s^5 -10 (26394629299\\&
  + 18663820476 \sqrt{2}) s^6 -4 (81220195657 + 57431327313 \sqrt{2}) s^7 -3 (1249507541447\\&
   + 883535234030 \sqrt{2}) s^8 -8 (1952782975722 + 1380826157423 \sqrt{2}) s^9\\& -
 4 (6888547193281 + 4870938333736 \sqrt{2}) s^{10} -8 (1765623286060\\& + 1248484278899 \sqrt{2}) s^{11} +
 4 (1060524382091 + 749904091846 \sqrt{2}) s^{12}\\& -8 (10187503437336 + 7203652478171 \sqrt{2}) s^{13} -
 4 (67196727862873\\& + 47515253841724 \sqrt{2}) s^{14} - 8 (30613232635258 + 21646822899983 \sqrt{2}) s^{15}\\& + (88705981222713 +62724554067730 \sqrt{2}) s^{16} -4 (26559971292821\\& + 18780806113083 \sqrt{2}) s^{17} -
 2 (758983649767561 + 536682308197872 \sqrt{2}) s^{18}\\& -12 (209656467942337 + 148249492741409 \sqrt{2}) s^{19} -
 6 (69368628477445\\& + 49051113886098 \sqrt{2}) s^{20} +4 (989266483060929 +
    699517028837821 \sqrt{2}) s^{21}\\& + (5907226373184438 +4177038704528408 \sqrt{2}) s^{22} +
 4 (711859947412323 \\&+ 503362800787333 \sqrt{2}) s^{23} - 5 (444135975055873 + 314053691717698 \sqrt{2}) s^{24}\\& -
 16 (296731916917516 + 209821170572357 \sqrt{2}) s^{25} -8 (435586202647283 \\&+ 308004476695976 \sqrt{2}) s^{26} -
 48 (15617742413006 + 11043341740883 \sqrt{2}) s^{27}\\& +5522071523832 (99 + 70 \sqrt{2}) s^{28} +
 48 (1935980053846 + 1368381215477 \sqrt{2}) s^{29}\\& -8 (41064076939363 + 29020973516804 \sqrt{2}) s^{30} -
 16 (544765504924\\& + 395911997797 \sqrt{2}) s^{31} +5 (59087344421887 + 41797084198718 \sqrt{2}) s^{32}\\& +
 4 (50008595927637 + 35387097717353 \sqrt{2}) s^{33} + (31251260699478 \\&+21884874985472 \sqrt{2}) s^{34} -
 4 (295092871209 + 172974346519 \sqrt{2}) s^{35}\\& +6 (2390136237115 + 1692584019198 \sqrt{2}) s^{36} +
 12 (489345890057\\& + 338456433011 \sqrt{2}) s^{37} - 2 (4935848951321 + 3462791906388 \sqrt{2}) s^{38}\\& -
 4 (1948144076339 + 1378504040823 \sqrt{2}) s^{39} + (1299188921913\\& + 915465226450 \sqrt{2}) s^{40} +
 8 (42096163298 + 28662109097 \sqrt{2}) s^{41} - 4 (226347584393\\& + 157627787656 \sqrt{2}) s^{42} +
 8 (8180322816 + 5416847189 \sqrt{2}) s^{43} -4 (3012605429\\& + 2168492186 \sqrt{2}) s^{44} -
 8 (2181018220 + 1605907699 \sqrt{2}) s^{45} -4 (2924338961\\& + 1819315324 \sqrt{2}) s^{46} -
 8 (1976638638 + 1468141303 \sqrt{2}) s^{47} - 3 (632591047\\& + 402233390 \sqrt{2}) s^{48} +
 4 (661956497 + 465143907 \sqrt{2}) s^{49} -10 (25294979\\& + 16934064 \sqrt{2}) s^{50} +
 4 (10999999 + 5558715 \sqrt{2}) s^{51} +6 (-221163 + 364306 \sqrt{2}) s^{52}\\& -
 20 (103489 + 96257 \sqrt{2}) s^{53} + 6 (16735 + 15484 \sqrt{2}) s^{54} -68 (439 + 333 \sqrt{2}) s^{55}\\& + 289 (1 + 2 \sqrt{2}) s^{56}.
\end{split}
\end{equation}
\begin{equation*}\begin{split}
\Delta(s)=&11520 s^2 (1-s^2)^2 (1 + s^2)^4 (1+2 s- s^2)^4 (390627 +
   276215 \sqrt{2} + 8 (479697\\& + 339197 \sqrt{2}) s +
   6 (799387 + 565252 \sqrt{2}) s^2 +
   6 (77084001 + 54506620 \sqrt{2}) s^3\\&  + (4935170703 +
      3489692688 \sqrt{2}) s^4 +
   2 (11909289353  + 8421139206 \sqrt{2}) s^5 \\& +
   2 (55180182851 + 39018281590 \sqrt{2}) s^6 + (575505934672 +
      406944145148 \sqrt{2}) s^7 \\& + (2113297274107 +
      1494326829313 \sqrt{2}) s^8 +
   4 (851541875150\\&  + 602131054387 \sqrt{2}) s^9 +
   4 (409164848135 + 289323230867 \sqrt{2}) s^{10}\\& + (7497187192104 +
      5301311414788 \sqrt{2}) s^{11} + (65072706732424\\& +
      46013348464164 \sqrt{2}) s^{12} +
   4 (44296248665194 + 31322183415247 \sqrt{2}) s^{13}\\& +
   4 (30236914005049 + 21380706750463 \sqrt{2}) s^{14} -
   36 (4037437677610\\& + 2854892242091 \sqrt{2}) s^{15} +
   100 (4180862840265 + 2956309394084 \sqrt{2}) s^{16} \\&+
   4 (596310623831302 + 421655642725771 \sqrt{2}) s^{17} +
   4 (687974927887557\\& + 486471275181883 \sqrt{2}) s^{18} -
   4 (353574626119520 + 250014848252183 \sqrt{2}) s^{19}\\& -
   4 (1050936777473819 + 743124176276455 \sqrt{2}) s^{20} +
   4 (1257191440592140\\& + 888966501074347 \sqrt{2}) s^{21} +
   4 (4646553272535371 + 3285620641347035 \sqrt{2}) s^{22}\\& +
   4 (553164210565510 + 391127178493017 \sqrt{2}) s^{23} -
   4 (12321436857957423 \\&+ 8712592044379198 \sqrt{2}) s^{24} -
   4 (14690543253852922 + 10387693072535489 \sqrt{2}) s^{25}\\& +
   20 (1124435991269047 +
      795090624671315 \sqrt{2}) s^{26} + (92504088406980200\\& +
      65409554729692564 \sqrt{2}) s^{27} + (78075824869764376 +
      55208610014942388 \sqrt{2}) s^{28}\\& +
   4 (13448850623812426 + 9510003149956667 \sqrt{2}) s^{29} -
   4 (6093684369943651\\& + 4309251971192341 \sqrt{2}) s^{30} -
   4 (57075032265120002 + 40358418086411519 \sqrt{2}) s^{31}\\&-
   2 (123045810817815239 + 87005122793249893 \sqrt{2}) s^{32} +
   4 (30386841536227474\\& + 21486947363073227 \sqrt{2}) s^{33} +
   20 (17726778591022006 \\&+ 12534545996157399 \sqrt{2}) s^{34} +
   28 (2993828147077073 + 2116872392697175 \sqrt{2}) s^{35}\\& -
   1895471885427634 (99 + 70 \sqrt{2}) s^{36} -
   28 (2322785292016873\\& + 1642349230904605 \sqrt{2}) s^{37} +
   20 (4972149139239326 + 3515200883825361 \sqrt{2}) s^{38}\\& +
   4 (5699952795134566 + 4031571485532787 \sqrt{2}) s^{39} -
   2 (38934011109465679 \\&+ 27526064428059847 \sqrt{2}) s^{40} -
   4 (7641926710147478 + 5405717188956199 \sqrt{2}) s^{41}\\& +
   4 (10157306186189269 + 7182519922073081 \sqrt{2}) s^{42} +
   4 (6366239451447214 \\&+ 4502096260405507 \sqrt{2}) s^{43} -
   4 (4606585487865346 + 3258051987873957 \sqrt{2}) s^{44}\\& -
   4 (4944939535256530 + 3495766010406259 \sqrt{2}) s^{45} +
   20 (157748975738447\\& +
      111504806546105 \sqrt{2}) s^{46} + (9945392349476168 +
      7030334533516124 \sqrt{2}) s^{47}\\& +
   4 (567617367920337 + 401810586777218 \sqrt{2}) s^{48} -
   4 (526303468130270\\
\end{split}\end{equation*}
\begin{equation}\begin{split}\label{delta}
   & + 372132796342383 \sqrt{2}) s^{49} -
   4 (313483174003229 + 221833702990975 \sqrt{2}) s^{50}\\& -
   4 (58017265637300 + 40979048784853 \sqrt{2}) s^{51} -
   4 (9608880993619\\& + 6756592336885 \sqrt{2}) s^{52} +
   4 (4653018198760 + 3277425508217 \sqrt{2}) s^{53}\\& +
   4 (12813482207997 + 9035681451337 \sqrt{2}) s^{54} +
   4 (9878641021618 \\&+ 7006765991651 \sqrt{2}) s^{55} +
   100 (196128025785 + 138532632416 \sqrt{2}) s^{56}\\& +
   36 (202968071090 + 143374448909 \sqrt{2}) s^{57} +
   4 (560290131089 + 395094153877 \sqrt{2}) s^{58}\\& + (616736716984 +
      442490670428 \sqrt{2}) s^{59} +
   4 (26308904186 + 18016329519 \sqrt{2}) s^{60}\\& -
   4 (3433626786 + 2360325463 \sqrt{2}) s^{61} + (914643580 +
      587708132 \sqrt{2}) s^{62}+
   4 (532792490\\&  + 407460587 \sqrt{2}) s^{63} + (161214947 +
      37758907 \sqrt{2}) s^{64} - 4 (30500828 + 16376993 \sqrt{2}) s^{65}\\&  -
   2 (1612349 + 3130730 \sqrt{2}) s^{66} -
   2 (1817833 + 855774 \sqrt{2}) s^{67} -
   3 (120619 + 144636 \sqrt{2}) s^{68}\\&  + 6 (2799 + 4760 \sqrt{2}) s^{69} -
   6 (853 + 632 \sqrt{2}) s^{70} -
   8 (57 + 23 \sqrt{2}) s^{71} + (27 + 5 \sqrt{2}) s^{72}).
\end{split}\end{equation}
\begin{equation*}\begin{split}
\phi_0(s)=&(1 + s^2)^4\left(3(9073 + 6417\sqrt{2}) + 6 (248308 + 175577\sqrt{2}) s +
 6 (3557810\right. \\&+ 2515683\sqrt{2}) s^2 +
 2 (50149748 + 35462575\sqrt{2}) s^3 + (172695877 +
    122185565\sqrt{2}) s^4\\& - 4 (1479014 + 1038139\sqrt{2}) s^5 -
 16 (5633018 + 4005649\sqrt{2}) s^6 +
 12 (180776524\\& + 127860795\sqrt{2}) s^7 + (5321803073 +
    3762624109\sqrt{2}) s^8 -
 2 (5883459724 \\&+ 4161916049\sqrt{2}) s^9 -
 2 (23816259518 + 16839545197\sqrt{2}) s^{10} -
 2 (10436003836 \\&+ 7356932573\sqrt{2}) s^{11} +
 3 (3085909949 + 2203596401\sqrt{2}) s^{12} -
 408 (458694391 \\&+ 325792078\sqrt{2}) s^{13} -
 480 (843450106 + 597011159\sqrt{2}) s^{14} -
 16 (17970722716\\& + 12543905561\sqrt{2}) s^{15} -
 2 (77639365337 + 55037084353\sqrt{2}) s^{16} +
 4 (81946369868\\& + 56878948051\sqrt{2}) s^{17} +
 4 (436036977998 + 308657296237\sqrt{2}) s^{18}\\& +
 12 (142872462716 + 101283760637\sqrt{2}) s^{19} -
 2 (223829242351\\& + 159679813967\sqrt{2}) s^{20} -
 24 (48478819642 + 34118961059\sqrt{2}) s^{21}\\& -
 64 (16799592844 + 11824003343\sqrt{2}) s^{22} -
 216 (6072830252+ 4383872847\sqrt{2}) s^{23}\\&  +
 2 (345671595401 + 248399077117\sqrt{2}) s^{24}+
 12 (176444104316\\&  + 126324762737\sqrt{2}) s^{25} -
 4 (64195344670 + 49431640397\sqrt{2}) s^{26} -
 4 (453920800364\\&  + 319842371881\sqrt{2}) s^{27} -
 2 (-825206959 + 1489184029\sqrt{2}) s^{28} +
 16 (62848691203\\&  + 44279656499\sqrt{2}) s^{29} +
 288 (431735606 + 338801465\sqrt{2}) s^{30} -
 48 (4754631844\\&  + 3458016283\sqrt{2}) s^{31} -
 3 (7584326715 + 5648352523\sqrt{2}) s^{32} +
 2 (10402471756\\&  + 7128027983\sqrt{2}) s^{33} -
 2 (1296277202 + 51451507\sqrt{2}) s^{34} -
 2 (1125780572\\&  + 1118430493\sqrt{2}) s^{35} + (-108689239 -
    157478351\sqrt{2}) s^{36} + 12 (11275830 \\
\end{split}\end{equation*}
\begin{equation*}\begin{split}
& + 11961031\sqrt{2}) s^{37} +
 16 (30370474 + 23716733\sqrt{2}) s^{38}+
 4 (8030852 - 552287\sqrt{2}) s^{39}\\&  + (-76238915 -
    54941431\sqrt{2}) s^{40} + 2 (-254492 + 2638187\sqrt{2}) s^{41} -
 6 (-374878\\&\left. + 40707\sqrt{2}) s^{42} + 6 (-30820 + 11749\sqrt{2}) s^{43} -
 3 (6711 + 7171\sqrt{2}) s^{44} + 504 s^{45}\right),\\
\phi_1(s)=&4 (1 + s^2)^2\left(-15 (181 + 128\sqrt{2}) - 6 (49867 + 35260\sqrt{2}) s -
 6 (969469 + 685511\sqrt{2}) s^2\right.\\& -
 6 (5305887 + 3752018\sqrt{2}) s^3 -
 2 (35136998 + 24857527\sqrt{2}) s^4 -
 6 (6092743\\& + 4309484\sqrt{2}) s^5 -
 2 (50201473 + 35440739\sqrt{2}) s^6 -
 6 (241851595 \\&+ 171033306\sqrt{2}) s^7 -
 130 (24810941 + 17550409\sqrt{2}) s^8 +
 6 (631987763\\& + 446728396\sqrt{2}) s^9 +
 2 (8720152075 + 6169934033\sqrt{2}) s^{10} +
 162 (48585445\\& + 34379702\sqrt{2}) s^{11} -
 2 (3203577146 + 2286774949\sqrt{2}) s^{12} +
 6 (6127189575 \\&+ 4343310236\sqrt{2}) s^{13} +
 62 (1614485699 + 1142045929\sqrt{2}) s^{14} +
 6 (28173421195\\& + 19894422154\sqrt{2}) s^{15} + (378349798135 +
    267994051562\sqrt{2}) s^{16} +
 12 (44187518665\\& + 31196995572\sqrt{2}) s^{17} +
 4 (34057927733 + 24006160303\sqrt{2}) s^{18} -
 12 (69608991835\\& + 49077924346\sqrt{2}) s^{19} -
 4 (476169765358 + 337108293995\sqrt{2}) s^{20}\\& -
 12 (157144547827 + 111432468284\sqrt{2}) s^{21} -
 36 (6673644773 + 4512236367\sqrt{2}) s^{22}\\& +
 12 (154899577273 + 109536852350\sqrt{2}) s^{23} +
 4 (691985193329\\& + 486697268587\sqrt{2}) s^{24} +
 12 (76932577283 + 55858249708\sqrt{2}) s^{25}\\& -
 12 (122933652271 + 87942310205\sqrt{2}) s^{26} -
 36 (30535411011\\& + 21874380250\sqrt{2}) s^{27} +
 4 (53003545238 + 43674114979\sqrt{2}) s^{28} +
 4236 (110547271\\& + 74223484\sqrt{2}) s^{29} +
 4 (34862053685 + 24734731711\sqrt{2}) s^{30} -
 12 (10988525013 \\&+ 7217628758\sqrt{2}) s^{31} + (-126949572797 -
    94000263628\sqrt{2}) s^{32} -
 102 (898095143\\& + 600037164\sqrt{2}) s^{33} -
 2 (18438488419 + 15161863385\sqrt{2}) s^{34} +
 6 (3567359045 \\&+ 2789149926\sqrt{2}) s^{35} +
 2 (11815170466 + 8692694045\sqrt{2}) s^{36} +
 18 (210070431\\& + 151822220\sqrt{2}) s^{37} -
 2 (2462862965 + 2172889663\sqrt{2}) s^{38} +
 114 (-4403149\\& + 230218\sqrt{2}) s^{39} +
 2 (416736491 + 321184675\sqrt{2}) s^{40} -
 6 (473705 + 8692964\sqrt{2}) s^{41}\\& -
 2 (109512961 + 82145939\sqrt{2}) s^{42} +
 18 (-290135 + 355182\sqrt{2}) s^{43} +
 2 (13584430 \\&+ 9831887\sqrt{2}) s^{44} -
 6 (61253 + 375252\sqrt{2}) s^{45} + 30 (-16793 + 6869\sqrt{2}) s^{46}\\&\left. -
 6 (-8879 + 1150\sqrt{2}) s^{47} + 705 (1 + 2\sqrt{2}) s^{48}\right),\\
 \phi_2(s)=&16s (1 - s^2)\left(135 (99 + 70\sqrt{2}) + 15 (22937 + 16219\sqrt{2}) s +
 24 (92181 + 65185\sqrt{2}) s^2\right.\\& + 3 (2699027 + 1909063\sqrt{2}) s^3 +
 6 (1435676 + 1015295\sqrt{2}) s^4 -(44085175\\&+
    31163579\sqrt{2}) s^5 -
 120 (444303 + 313891\sqrt{2}) s^6 + (817427887 +
    578000057\sqrt{2}) s^7\\
& +
 6 (474651565 + 335595393\sqrt{2}) s^8 + (483511397 +
    341440219\sqrt{2}) s^9\\
\end{split}\end{equation*}
\begin{equation*}\begin{split}
 & -
 24 (516591181 + 365240565\sqrt{2}) s^{10} + (-16014253381 -
    11311724237\sqrt{2}) s^{11} \\&+
 162 (202807832 + 143424883\sqrt{2}) s^{12} +
 3 (42950321273 + 30343273537\sqrt{2}) s^{13}\\& +
 24 (4670621819 + 3302349499\sqrt{2}) s^{14} -
 3 (75991971377 + 53641650067\sqrt{2}) s^{15}\\& -
 3 (203907137827 + 144265964148\sqrt{2}) s^{16} -
 6 (60672645767 + 42953256077\sqrt{2}) s^{17}\\& +
 48 (7856365753 + 5575559445\sqrt{2}) s^{18} +
 2 (375512954597 + 265234148329\sqrt{2}) s^{19}\\& +
 12 (41689439668 + 29404697227\sqrt{2}) s^{20} -
 2 (117746827819 + 83080481543\sqrt{2}) s^{21}\\& -
 48 (18117052119 + 12784457243\sqrt{2}) s^{22} -
 2 (54218402077 + 37855534259\sqrt{2}) s^{23}\\& +
 12 (87082202767 + 61375066769\sqrt{2}) s^{24} +
 2 (183341325277 \\&+ 129927057659\sqrt{2}) s^{25} -
 48 (12228124589 + 8643200357\sqrt{2}) s^{26}\\& -
 2 (39611650205 + 29751577981\sqrt{2}) s^{27} +
 12 (32056183104 + 22979177857\sqrt{2}) s^{28}\\& +
 2 (9897688051 + 10603248563\sqrt{2}) s^{29} -
 48 (6624564597 + 4938170005\sqrt{2}) s^{30}\\& -
 6 (36512886953 + 25854919283\sqrt{2}) s^{31} +
 3 (17631777425 + 16236292934\sqrt{2}) s^{32}\\& +
 3 (44707057633 + 28989406883\sqrt{2}) s^{33} +
 24 (62397481 + 41714885\sqrt{2}) s^{34}\\& -
 3 (14738319953 + 9503890813\sqrt{2}) s^{35} -
 18 (-72163508 + 22334359\sqrt{2}) s^{36}\\& + (6278132509 +
    4122349817\sqrt{2}) s^{37} -
 24 (45993239 + 9990027\sqrt{2}) s^{38}\\& - (274788485+
    311602723\sqrt{2}) s^{39} +
 6 (47884229 + 21948621\sqrt{2}) s^{40} + (-17527951\\& +
    23723071\sqrt{2}) s^{41} -
 168 (126831 + 85351\sqrt{2}) s^{42}-(10964369+
    12180841\sqrt{2}) s^{43}\\& + 6 (-23720 + 224653\sqrt{2}) s^{44} +
 3 (698677 + 475277\sqrt{2}) s^{45} - 120 (461 + 1389\sqrt{2}) s^{46}\\& \left.+
 15 (-1577 + 1301\sqrt{2}) s^{47} + 3105 s^{48}\right),\\
 \phi_3(s)=&128 s^3(1 - s^2)^3 \left(-135 (99 + 70\sqrt{2}) - 30 (7525 + 5321\sqrt{2}) s -
 6 (152746\right.\\& + 108015\sqrt{2}) s^2 - 6 (227183 + 160757\sqrt{2}) s^3 +
 6 (432025 + 305483\sqrt{2}) s^4 \\&+
 2 (9742589 + 6888895\sqrt{2}) s^5 +
 2 (11052254 + 7811245\sqrt{2}) s^6 -
 2 (63240653\\& + 44719837\sqrt{2}) s^7 + (-397483285 -
    281030456\sqrt{2}) s^8 - 8 (10862657\\& + 7667785\sqrt{2}) s^9 +
 248 (5120762 + 3620671\sqrt{2}) s^{10} +
 24 (78243119\\& + 55256821\sqrt{2}) s^{11} -
 24 (65939253 + 46640495\sqrt{2}) s^{12} -
 8 (1174003525\\& + 829395239\sqrt{2}) s^{13} -
 168 (72000814 + 50921581\sqrt{2}) s^{14} -
 8 (65378599\\& + 48038387\sqrt{2}) s^{15} +
 6 (1806045881 + 1280295926\sqrt{2}) s^{16} +
 4 (476075279\\& + 336969379\sqrt{2}) s^{17} -
 4 (2738422394 + 1941231319\sqrt{2}) s^{18} -
 12 (307131301\\& + 218227703\sqrt{2}) s^{19} +
 884115356 (7 + 5\sqrt{2}) s^{20} +
 12 (145879621\\& + 105351527\sqrt{2}) s^{21} -
 4 (668567654 + 492333001\sqrt{2}) s^{22} -
 4 (44260439\\& + 34698991\sqrt{2}) s^{23} +
 6 (442887421 + 326085004\sqrt{2}) s^{24} +
 8 (252892879\\& + 179298383\sqrt{2}) s^{25} -
 168 (940754 + 1179539\sqrt{2}) s^{26} -
 8 (111015515\\
\end{split}\end{equation*}
\begin{equation}\begin{split}\label{phi}
& + 70118089\sqrt{2}) s^{27} -
 24 (1683253 + 1661295\sqrt{2}) s^{28} +
 24 (10113841 \\&+ 6593051\sqrt{2}) s^{29} -
 248 (61498 + 6911\sqrt{2}) s^{30} -
 8 (1913143 + 1275275\sqrt{2}) s^{31}\\& + (6581375 +
    1814806\sqrt{2}) s^{32} + 2 (-47467 + 418153\sqrt{2}) s^{33} -
 2 (598846\\& + 344525\sqrt{2}) s^{34} - 2 (-71011 + 19375\sqrt{2}) s^{35} +
 6 (-2855 + 1067\sqrt{2}) s^{36} + 6 (14863 \\&\left.+ 12133\sqrt{2}) s^{37} -
 6 (246 + 1265\sqrt{2}) s^{38} + 30 (-35 + 29\sqrt{2}) s^{39}+ 135 s^{40}\right).
\end{split}\end{equation}
\begin{equation*}\begin{split}
R=&1592524800s^6(1-
   2 s-s^2)^4(1-s^2)^6 (1+s^2)^{10} (1 + 2 s -s^2)^{12}(-2835 (3031967863973 \\&+ 2143925036955\sqrt{2}) -
   990 (494121750080087 + 349396840213394\sqrt{2}) s\\& -
   36 (264097013248510342 + 186744788959135145\sqrt{2}) s^2 -
   36 (1887468190562069883\\& +
      1334641556820342334\sqrt{2}) s^3 + (91473639307849057536\\& +
      64681630654392396069\sqrt{2}) s^4 +
   6 (1033939717614918120928 \\&+
      731105785663612644769\sqrt{2}) s^5 + (89674919743321102482210\\& +
      63409743852861765304554\sqrt{2}) s^6 +
   6 (169513037149998265360064 \\&+
      119863818068290925008325 \sqrt{2}) s^7 + (10006126095391353411154527 \\&+
      7075399615458896897729482\sqrt{2}) s^8 +
   10 (8104987543069728084647223 \\&+
      5731091653137099645934904\sqrt{2}) s^9 +
   4 (137633401828503741973000623 \\&+
      97321511750707965418986469 \sqrt{2}) s^{10} + (3460794542183976146715930922\\& +
      2447151289071682737189213348 \sqrt{2}) s^{11} + (20668314075174741603761104223\\& +
      14614705038249426372214003364\sqrt{2}) s^{12} +
   2 (52438474590388672317874236492\\& +
      37079600977942295239659370631 \sqrt{2}) s^{13} + (392207152806294465505251165866 \\&+
      277332337379199269450917653786 \sqrt{2}) s^{14} + (861676195948774227993715759992\\& +
      609297081342406546436605709942 \sqrt{2}) s^{15} - (82709169427171204680364507426\\& +
      58484214568259835247263300487\sqrt{2}) s^{16} -
   8 (859317064411726826865506732551\\& +
      607628923434849286255795764873\sqrt{2}) s^{17} -
   4 (2963221875527768270897014277738\\& +
      2095314282346004603177512679185 \sqrt{2}) s^{18}\\& + (68135969004058954533141035965514 +
      48179405725486499611989334991320 \sqrt{2}) s^{19}\\& + (431257166833013944563734906365669 +
      304944867102922411738399492775765\sqrt{2}) s^{20}\\& +
   84 (10888444792580818055548594553776+
      7699293149409247357101383674411\sqrt{2}) s^{21}\\& -
   4 (141558998343223234019005315458339+
      100097327666468394727161624615559\sqrt{2}) s^{22}\\& -
   4 (1928324031568714651378566974620944 +
      1363530999047220272429435276458301\sqrt{2}) s^{23}\\& -
   2 (6265994871098688008249222536383309 +
      4430727464234009036394322559590834\sqrt{2}) s^{24}\\& +
   4 (8024321484941658643497407645143903 +
      5674052136423153516233390867435596\sqrt{2}) s^{25}
   \end{split}\end{equation*}
\begin{equation*}\begin{split}
& + 8 (22804167846315131553884458068857961\\& +
      16124981723445656450173223328696407\sqrt{2}) s^{26}\\& +
   4 (68957082018022061880225467228576605 \\&+
      48760020305780337100626585751556542\sqrt{2})s^{27}\\
&-2 (179122109008759510698238339957867105 +
      126658457940529824850022579834649264\sqrt{2}) s^{28}\\& -
   4 (544901907499494053979665206957422116 +
      385303833874377100174703821484811769\sqrt{2}) s^{29}\\& -
   4 (553746174130839028616519403815423659 +
      391557674784023040929516385035033643 \sqrt{2}) s^{30}\\&
   + (8583949559758434627953359230241959344 +
      6069768943068468431403184426131130348 \sqrt{2}) s^{31}\\& + (32551763897526059043342794151904010249 +
      23017572991524116362810082690825248637\sqrt{2}) s^{32}\\& +
   6 (4590958783730906581032175004775828945 +
      3246298088124068533878558739427807374\sqrt{2}) s^{33}\\& -
   4 (23316892051255241762357713028810541734 +
      16487532485637289453025514794708906195\sqrt{2}) s^{34}\\& -
   8 (34465659637331272162211079464195995507 +
      24370901647624426891754321315661625435 \sqrt{2}) s^{35}\\& - (2723250446710848541827911263024492894 +
      1925628857738536871192812826239675389\sqrt{2}) s^{36}\\& +
   2 (705213307040225347800486097601813094368 \\&+
      498661111591134187097108239264342367355\sqrt{2}) s^{37}\\&
      +
   2 (1408268502554507723228915929808546566233\\& +
      995796207887717220126708149713141425965\sqrt{2}) s^{38}\\& -
   2 (399810574331351134670250706336099547328\\& +
      282708768299786592411122476175966132089\sqrt{2}) s^{39}\\& -
   3 (4080629009669078909103719801524127061139\\& +
      2885440444243551510636026114362030566698\sqrt{2}) s^{40}\\& -
   2 (6704344098538994865913155724400021847445 \\&+
      4740687175484934199510433686666190760352\sqrt{2}) s^{41}\\& +
   12 (2838062956814256154989343568835039239249\\& +
      2006813562197704336062964082579372588035\sqrt{2}) s^{42}\\& +
   2 (55146666252281557103187501005312466461525 \\&+
      38994581666819619196799428839233983870858 \sqrt{2}) s^{43}\\& + (48520698688951069064430983502057417948015\\& +
      34309315070866522047293733097812955006508\sqrt{2}) s^{44}\\& -
   2 (137662711194989407452021833499901554937396\\& +
      97342236602502250151809027615892665533785\sqrt{2}) s^{45}\\& -
   2 (224920549525140898848123765835196520012555\\& +
      159042845797431837933389777387675571431747 \sqrt{2}) s^{46} \\&+ (348633316915199240520396333533279933107448\\& +
      246520982538296003016515988313057164094774 \sqrt{2}) s^{47}\\& + (1808057555754074019384020124549628754785460 \\&+
      1278489758449280059730907204845548646765691\sqrt{2}) s^{48}\\&
\end{split}\end{equation*}
\begin{equation*}\begin{split}
      & +
   4 (368923151311216172869052309346396830397527\\& +
      260868062028871733816051210840502460137720\sqrt{2}) s^{49}\\& -
   4 (559501961162572999897094631072157873480178\\&+
      395627630825227833323672837017168403428363\sqrt{2}) s^{50}\\&  -
   10 (553841577836044910024479322103243494018555\\& +
      391625135390924397779040984326275906530788 \sqrt{2}) s^{51}\\& - (2045276596487240170762272145281835657325635\\& +
      1446228950778269389094529602407987347133701\sqrt{2}) s^{52}\\& +
   176 (46967511418618196903008187033308902146992\\& +
      33211045819561518759124306024600747330475\sqrt{2}) s^{53}\\& +
   16 (1192497130261441072188951728666949847632771\\& +
      843222807353362958632880162004240375400535\sqrt{2}) s^{54}\\& +
   16 (1389359969660933778464225671898620543677008\\& +
      982425856056382041537158089473326045582557\sqrt{2}) s^{55}\\& -
   8 (838122919185892527274133489845351330483395\\& +
      592642399624210068027836298385595389329302\sqrt{2}) s^{56}\\& -
   16 (6091256618544973305563841416405315008672703\\& +
      4307168860920589408873343433257566309242292\sqrt{2}) s^{57}\\& -
   32 (5411734908250012834926087222398651552108325\\& +
      3826674451607542313167464249383951379956419\sqrt{2}) s^{58}\\& -
   16 (654266571292426610666004643859371152948533\\&
   +462636329264552157816871731797324913339526\sqrt{2}) s^{59}\\& +
   8 (50467895677347008102385082117331830823377673 \\&+
      35686191265667353638096273874322467387369512\sqrt{2}) s^{60}\\& +
   16 (31152421072480347888275764512370652430095428 \\&+
      22028088190729542029599930904462653536364497\sqrt{2}) s^{61}\\& -
   16 (13631042376246301069320861299717602964631181 \\&+
      9638602498885009812404037554930643760478749\sqrt{2}) s^{62}\\& -
   16 (57304288603146447630028894716165322434213388\\& +
      40520251062355731992988652346701301092345923\sqrt{2}) s^{63}\\& -
   2 (102117606690558265846057644987517551760622303\\& +
      72208052169433892332954062013174038023669357\sqrt{2}) s^{64}\\& +
   4 (275547160674870583863542603244430999547720575\\& +
      194841265849898816971328178664681423629911298\sqrt{2}) s^{65}\\
 \end{split}\end{equation*}
 \begin{equation*}\begin{split}
     &+
   8 (39096218401699640940314987864652688490277658\\&  +
      27645201150591964035520989634910580162359123\sqrt{2}) s^{66}\\&  -
   8 (248823455695407456938092904806022984034978333\\&  +
      175944752840490833496393411729123446628361934\sqrt{2}) s^{67}\\&  -
   2 (487355044425398274840976061693501454405013636\\&  +
      344612056758679404147551600451183335890224231\sqrt{2}) s^{68}\\&  +
   4 (926884481809837072763076995491688304690655904 \\& +
      655406302464309986221703608182601304670058659\sqrt{2}) s^{69}\\&  +
   4 (908910323347687626854601596159677041713868625\\&  +
      642696653129626948507300627237708879667556309\sqrt{2}) s^{70}\\&  -
   4 (1196861311364319796440101462110466161301221184\\&  +
      846308749405558922288041376080871031337748593\sqrt{2}) s^{71}\\&  -
   2 (4370642229792309375302233535077641266415288257 \\& +
      3090510758826440685284336345456724878746780454\sqrt{2}) s^{72}\\&  +
   4 (742178628988608027044648050282398599526634131\\&  +
      524799541409651412038857711303962167304898472\sqrt{2}) s^{73}\\&  +
   8 (1899031569965905511950288461135498336631260423\\&  +
      1342818100810172926296088091641041081535896541\sqrt{2}) s^{74} \\& +
   4 (1109034759240925026517029633415896620508287749 \\& +
      784205998830833511437172982642545978164991010\sqrt{2}) s^{75}\\&  -
   2 (9200827589669696556493179498987053117569031857\\&  +
      6505967581183244534882034362639750248751551564\sqrt{2}) s^{76}\\&  -
   4 (4123813675892972141831298193166214817475719668\\&  +
      2915976614573959549011839477736595062496717097\sqrt{2}) s^{77}\\&  +
   4 (3218327841078439770177877982623808637413244357\\&  +
      2275701440507850058731641571231093767568360605\sqrt{2}) s^{78} \\& +
   4 (6474662136774769935869206307406366275654187772\\&  +
      4578277502805641498093994934792147279406087115\sqrt{2}) s^{79}\\&  +
   2 (68257673563881129361815558434638545541803794\\&  +
      48265463844889811061007868654309911478706821\sqrt{2}) s^{80}\\&  -
   16 (1618769403618776813256690713318428482905353756 \\& +
      1144642822476229107838948466080593948419910075\sqrt{2}) s^{81}\\&  -
   40 (321779310488316287482684957574837967568316290\\&  +
      227532332491797870737606536154286728949815481\sqrt{2}) s^{82} \\
   \end{split}\end{equation*}
\begin{equation*}\begin{split}
& +
   4 (4244960055484698546625343255302246243830442701 \\& +
      3001640041099373320970938088742412509460168588\sqrt{2}) s^{83}\\&  +
   2 (9227766719396036865781224893491917577810200901\\&  +
      6525016422492055093151200471974450899460941913\sqrt{2}) s^{84}\\& -
   8 (667933797085487368669784515705909725074404016\\&  +
      472300517302704336693808970077219825535130707\sqrt{2}) s^{85}\\&  -
   56 (281471387639097128634866312657006445238699607\\&  +
      199030326909576345760982550911794336060680587\sqrt{2}) s^{86}\\& -
   8 (320871332305997796526009811090292964059323824 \\&+
      226890294962631919985093066515884987870221247 \sqrt{2}) s^{87} \\&+ (8997356032557396823349162765851876929920488388\\& +
      6362091463372797131467446249603766029197382696\sqrt{2}) s^{88}\\& +
   8 (584576628642992981824256468003903131231258549 \\&+
      413358098238478497395117328388637855005147828\sqrt{2}) s^{89}\\& -
   11492643374198789877538186162577948020400 (275807 +
      195025\sqrt{2}) s^{90}\\& -
   8 (397612129516086640926688362756920605286485649\\& +
      281154233065565647340161383243516692248450022\sqrt{2}) s^{91}\\& +
   4 (95237754882911966890847578071256580878811397\\& +
      67343262313143685989277571181282255151665376\sqrt{2}) s^{92}\\& +
   8 (150342788992712137681474358048614714997516724\\& +
      106308405600736322731062071266388377592077853\sqrt{2}) s^{93}\\& +
   56 (3572718830655221336462828192295434024426993 \\&+
      2526293711120965836357216682676637952715537\sqrt{2}) s^{94}\\& -
   8 (26510508588920727843436266654138745340861084\\& +
      18745760392825656576340714118083054813088807\sqrt{2}) s^{95}\\& -
   2 (91191390010770455781715671298694589254192499\\& +
      64482050219554461875872847736366611015994763\sqrt{2}) s^{96}\\& -
   4 (25807523343524341570223823677542497462221801\\& +
      18248674781565773288237601425107891655836062\sqrt{2}) s^{97}\\& +
   40 (4761681166060659774091307270835294525332010\\& +
      3367017040874203847830436742240909553614481\sqrt{2}) s^{98}\\& +
   16 (19075702446399399682438096865898832552016256\\& +
      13488558563269494185124875807749837985342825\sqrt{2}) s^{99}\\& -
   2 (31901810647619340788942353908692055069066506\\& +
      22557986640743595350657526883895403128862221\sqrt{2}) s^{100}
\end{split}\end{equation*}
\begin{equation*}\begin{split}
& -
   4 (90368921663392471024157215307960925677338272\\& +
      63900477346793441571706882120990031973540185\sqrt{2}) s^{101}\\& -
   4 (37899294091968153740310010089265403059024643\\& +
      26798847839655897612682683880340181619013055\sqrt{2}) s^{102}\\& +
   4 (46363104284733046710386902715860613939397568\\& +
      32783665455760343504819402655859664945781603\sqrt{2}) s^{103}\\& +
   2 (102212443221093411567538514007974453436955843 \\&+
      72275111680516800128662863382861132549431514\sqrt{2}) s^{104}\\& +
   4 (3157717134269744955491278169762441617977751\\& +
      2232843193556552235488407081094244749658160\sqrt{2}) s^{105}\\& -
   8 (11712750277284768467640043962762153390578377\\& +
      8282165138586416534901799925773252450652091\sqrt{2}) s^{106}\\& -
   4 (15473763779809266046119912983559576204117031 \\&+
      10941603302631389286079734213894497761215678\sqrt{2}) s^{107}\\& -
   2 (1294459770209478462946035042542231714123557\\& +
      915321301802033194723814955044172519824596\sqrt{2}) s^{108}\\& +
   4 (5275238940487828878209852725522666140996284\\& +
      3730157232761015667498781260634880665639507\sqrt{2}) s^{109}\\& +
   4 (4181827898473553932763745044629263222082425\\& +
      2956998868990143455760229333624832385197441\sqrt{2}) s^{110}\\& +
   4 (1068452965093402958576780291861034754682796\\& +
      755510332688474789351802502011931234057559\sqrt{2}) s^{111}\\& -
   2 (1970481324962626038595419510241475860125336 \\&+
      1393340709347639016829391045624087325860669\sqrt{2}) s^{112}\\& -
   8 (484003367310741962829096806259299729140367\\& +
      342242061986075321108432902792307347103484\sqrt{2}) s^{113}\\& -
   8 (29423782458836004816211136106712338561242\\& +
      20805755923090574928441470851368470891923\sqrt{2}) s^{114}\\& +
   4 (290430935189519983720631548961548502888325 \\&+
      205365682458181007910716345346056180867548\sqrt{2}) s^{115}\\& +
   2 (131858617655328255177195873702998520305297\\& +
      93238122227347888246496257331426176642907\sqrt{2}) s^{116}\\& -
   16 (24148954187053067521535647105186641190512\\& +
      17075888997891457567152158368688130984223\sqrt{2}) s^{117}\\
\end{split}\end{equation*}
\begin{equation*}\begin{split}
& -
   16 (12864202401983539972219820214195305762981\\& +
      9096364816352756085461063359316094360401\sqrt{2}) s^{118}\\& +
   16 (2470347782971620035247922973732236446672\\& +
      1746799524439161603082572668557082081797\sqrt{2}) s^{119}\\& +
   8 (7555994235812386878179738583739155333573 \\&+
      5342894997312472330827843363936638264938\sqrt{2}) s^{120}\\& +
   16 (1195906788113778530047687510995195324233\\& +
      845633802583155655810636400384706328924\sqrt{2}) s^{121} \\&+
   32 (76508457806988302640276334278884450875\\& +54099624180958170173605360397664822669\sqrt{2}) s^{122}\\& -
   16 (88566114665640051991720745962384800397 \\&+
      62625671952672073884076185536790518342\sqrt{2}) s^{123}\\& -
   8 (152059510122515045394283710769333587295\\& +
      107522314646935704360143466070567337448\sqrt{2}) s^{124}\\& +
   16 (35898682928259937531415352787101773092\\& +
      25384195676815597756452158641621302857\sqrt{2}) s^{125}\\& +
   16 (61567350846926717487214926371393479771\\& +
      43534696826003165508765667238140547115\sqrt{2}) s^{126}\\& +
   176 (2345942914819983764265346736428000508\\& +
      1658831925051651054003867933959430175 \sqrt{2}) s^{127}\\& + (43336238666833301134476841356523131165\\& +
      30643338726465461850323826934498053451\sqrt{2}) s^{128}\\& -
   10 (8514922379064982580571620281178297345\\& +
      6020956781386799275047229176392655038\sqrt{2}) s^{129}\\& -
   4 (24944580286048060859012466990956409278\\& +
      17638484474552141525324491504329856837\sqrt{2}) s^{130}\\& -
   4 (7756885731392681080587806510228784027\\& +
      5484948216226139333055212374004987830 \sqrt{2}) s^{131}\\& + (19632724820351354687032168344701759160 \\&+
      13882441257071862258897678762822350809\sqrt{2}) s^{132}\\& +
   2 (6897471813289108474323067357966105376\\& +
      4877248282036420842001750779372863287\sqrt{2}) s^{133}\\& -
   2 (1436528340390501798420589937894727955 \\&+
      1015779976235304303879435979517941503\sqrt{2}) s^{134}\\
\end{split}\end{equation*}
\begin{equation*}\begin{split}
& -
   2 (2378867322816064042377790409362503104\\& +
      1682112575681988416897069011278079885 \sqrt{2}) s^{135}\\& - (1051685161934051462239512031016213885\\& +
      743653484163891394030419291746341758\sqrt{2}) s^{136}\\& +
   2 (133489178395812978322740467412430375\\& +
      94390846949791660889965668749259608\sqrt{2}) s^{137}\\& +
   12 (7038115768340804615565444158141249 +
      4976712577227828513944311897282315\sqrt{2}) s^{138}\\& -
   2 (12353221467874782426385070241438655 +
      8735015509197805362787802889886102\sqrt{2}) s^{139}\\& -
   3 (1771302932839779314210265803822239 +
      1252519281163392964114853802499652\sqrt{2}) s^{140}\\& -
   2 (1866929568012947238721152811750372 +
      1320116699312849579758223301034389\sqrt{2}) s^{141}\\& -
   2 (3427145836469351311520472714240767 +
      2423351515777496608841389854855015\sqrt{2}) s^{142}\\& -
   2 (1569872633731246448520750011522868 +
      1110070862578362537179249562631345 \sqrt{2}) s^{143}\\& - (234918115589129498294927387590194 +
      166112205213674842615434410404211\sqrt{2}) s^{144}\\& +
   8 (25885878453745883749681184432507 +
      18304240380003017000855558894615\sqrt{2}) s^{145}\\& +
   4 (10287897841798495593016472656766 +
      7274533956586375248646981705595\sqrt{2}) s^{146}\\& -
   6 (3318108772505414269081326588245 + 2346278551464287614555275096876 \sqrt{2}) s^{147}\\& - (17408085980909913263302104086351 +
      12309224351534607411061815723287\sqrt{2}) s^{148}\\& -
   4 (2370888963624368896702542190736 +
      1676481637670569240550922522313\sqrt{2}) s^{149}\\& -
   4 (999354344596542874261475546259 +
      706652807556179007871323113207\sqrt{2}) s^{150}\\& -
   4 (290211198586825415278149869584 +
      205207773919393369075210449869\sqrt{2}) s^{151}\\& -
   2 (93810879567655075922368149405 +
      66335141609362506884646423986\sqrt{2}) s^{152}\\& +
   4 (613520691322167716862676495 +
      433504144715787866048613292\sqrt{2}) s^{153}\\& +
   8 (1670196414487280276187720361 +
      1181113199078069904287314743 \sqrt{2}) s^{154}\\& + (6870725049042884522123425588 +
      4858187093109324119993486584\sqrt{2}) s^{155}\\& +
   2 (1194988759874814120087945391 +
      844955532658988452313343984\sqrt{2}) s^{156}\\&  +
   4 (140663378908059751316111644 +
      99472991493151054068083799\sqrt{2}) s^{157}\\&  +
   4 (30556732439122570835137861 +
      21606214785229661073416709\sqrt{2}) s^{158}\\&  +
   84 (244342106371301578358524 +
      172725353379717653988511 \sqrt{2}) s^{159}\\&  + (3657065035674579079828669 +
      2587939869207976723772585 \sqrt{2}) s^{160}\\&  + (543547892383411937095486 +
      384029720430348123901220 \sqrt{2}) s^{161}\\&  + (86569277328988696951048 +
      61158633570988338739940\sqrt{2}) s^{162}\\& +
   8 (3907530873189777686151 +
      2767035484447202504777 \sqrt{2}) s^{163} \\&+ (6896091785017329236674 +
      4875888851842986202087\sqrt{2}) s^{164}
   \end{split}\end{equation*}
\begin{equation}\begin{split}\label{R}
&+
   2 (557751530937641790304 + 392387454235740346071\sqrt{2}) s^{165}\\& +
   2 (114039933336436795533 +
      81549854107382219057\sqrt{2}) s^{166} + (39106254067306623616\\& +
      27164853713716810662\sqrt{2}) s^{167} + (1811767716088461523 +
      1377174765913968586\sqrt{2}) s^{168}\\& + (64425268650110478 +
      29470574355106048\sqrt{2}) s^{169} +4 (5871290256731423\\& + 4791317270007981\sqrt{2}) s^{170} +
   10 (1022614051817477 +
      685427229162954\sqrt{2}) s^{171}\\& - (744861464245573 +
      480190440148932\sqrt{2}) s^{172} +
   6 (3612166972436\\& + 1766330498225\sqrt{2}) s^{173} +
   6 (898430050835 + 704750684491\sqrt{2}) s^{174} -
   6 (52998899228\\& + 42281392931\sqrt{2}) s^{175} + (4054225836 +
      3291918831\sqrt{2}) s^{176} +
   36 (21356183 \\&+ 23873616\sqrt{2}) s^{177} -
   36 (56842 + 1267655\sqrt{2}) s^{178} +
   990 (-1613 + 1156\sqrt{2}) s^{179}\\& + 2835 (27 + 5\sqrt{2}) s^{180}).
      \end{split}\end{equation}
\end{document}